\documentclass[11pt]{article}

\usepackage[utf8]{inputenc}
\usepackage{amsmath, amssymb, amsthm, bm}
\usepackage{physics}
\usepackage{yhmath}
\usepackage{extpfeil}
\usepackage{graphicx}
\usepackage{hyperref}
\usepackage{mathtools}
\usepackage{booktabs}
\usepackage{geometry}
\usepackage{algorithm}
\usepackage{algorithmic}
\usepackage{caption}
\usepackage{subcaption}
\usepackage{tikz-cd}
\usepackage{algorithm, algorithmic}
\usepackage{float}

\newtheorem{theorem}{Theorem}[section]
\newtheorem{lemma}[theorem]{Lemma}
\newtheorem{defn}[theorem]{Definition}
\newtheorem{corollary}[theorem]{Corollary}
\newtheorem{proposition}[theorem]{Proposition}
\newtheorem{fact}[theorem]{Fact}
\newtheorem{example}[theorem]{Example}
\newtheorem{remark}[theorem]{Remark}

\title{Higher Degree $t$-Hermitian Forms and Positivity-Preserving Contractions}
\author{Isaac Dobes}
\date{\today}

\begin{document}

\maketitle

\begin{abstract}
    In this article we introduce higher-degree $t$-Hermitian forms, a tubal analogue of ordinary Hermitian forms of arbitrary degree. Through a synthesis of multilinear matrix multiplication and the $t$-product on third-order tensors, we show that $t$-Hermitian forms are in bijection with odd order tubal tensors satisfying certain symmetry conditions, which we call $t$-conjugate partial symmetry. After applying the Fast Fourier Transform along the tubal mode of the corresponding tubal tensor, $t$-Hermitian forms decompose into a family of classical Hermitian forms. This decomposition enables us to characterize positivity of $t$-Hermitian forms in terms of the spectra of the conjugate partially symmetric Fourier slices of its corresponding tubal tensor, yielding a tubal analogue of the spectral theorem for classical higher degree Hermitian forms. Then, we study classical Hermitian forms induced by contractions along the tubal mode of $t$-conjugate partially symmetric tensors, characterizing when such contractions preserve positivity and deriving quantitative lower bounds for the positivity margin of the resulting classical Hermitian forms. 
\end{abstract}

\section{Introduction}
\subsection{Prior Work \& Paper Outline}
Hermitian forms occupy a significant slice of the linear-algebraic landscape because questions regarding their positivity may be translated into spectral questions. In the simplest case, a Hermitian quadratic form $h$ on $\mathbb{C}^n$ is uniquely represented by an $n\times n$ Hermitian matrix $H$, and positivity of $h$ is equivalent to positivity of the spectrum of $H$. Higher-degree analogues arise from mixed homogeneous polynomials of complex variables and their formal conjugates. In particular, a real-valued Hermitian form of bidegree $(k,k)$ may be represented by an order $2k$ conjugate partially-symmetric coefficient tensor \cite{d2011hermitian,jiang2016characterizing}, and likewise positivity of the higher-degree Hermitian form is equivalent to positivity of the tensor's spectrum \cite{chen2025hat}. In addition to Hermitian forms, such tensors have been investigated in the context of tensor decompositions, tensor rank, tensor spectral theory, convex optimization, and even quantum entanglement (see for example, \cite{ni2019hermitian,nie2020hermitian,huang2021decompositions,fu2021decompositions}). 

Separately, a substantial body of research has developed around the $t$-product $*_t$ for third-order tensors. Introduced by M. Kilmer, C. Martin and L. Perrone \cite{kilmer2000third}, and further developed with the help of K. Braman \cite{braman2010third}, the $t$-product can be thought of as a multiplication operation on tube-valued matrices. One particularly nice feature of the $t$-product is that after a Fourier change of basis is applied to the tubes, slicewise the $t$-product reduces to ordinary matrix multiplication \cite{kilmer2013third}. This central property has led to burgeoning tubal matrix theory with notions of $t$-eigenvalues, the $t$-determinant, $t$-quadratic forms, $t$-positive definiteness, the $t$-exponential map, $t$-Jordan normal form, and so on (see for example, \cite{gleich2013power,qi2021t,zheng2021t,miao2021t}).

The purpose of the this paper is to develop a higher-degree theory of tube-valued Hermitian forms, bringing together key aspects of these two lines of research. We refer to these objects as $t$-Hermitian forms. Intrinsically, a degree $k$ $t$-Hermitian form is defined as the diagonal of a separately symmetric, tube-valued Hermitian $(k,k)$-sesquilinear form. In coordinates, such forms correspond uniquely to higher-order tubal tensors satisfying a condition that we call $t$-conjugate partial symmetry. After a Fourier change of basis, every $t$-Hermitian form decomposes into a family of ordinary real-valued Hermitian forms of bidegree $(k,k)$. Thus, $t$-Hermitian forms may be viewed as Fourier-packaged families of classical Hermitian forms.

Using this Fourier decomposition, we characterize positivity of a $t$-Hermitian form in terms of the $\widehat{H}$-eigenvalues of the Fourier-transformed frontal slices of the corresponding tubal tensor. We also introduce normalized symmetric matricizations of these slices, which yield matrix-based sufficient conditions for $t$-Hermitian positivity. Additionally, we also consider classical Hermitian forms induced from tensor contractions of coefficient tensors of $t$-Hermitian forms and characterize when such contractions preserve positivity. 

The specific outline of this paper is as follows. In Section \ref{The Different Representations of Classical Hermitian Forms} we review the various different ways one may represent a degree $k$ Hermitian form in $n$ complex variables: as the diagonal of a $(k,k)$-sesquilinear forms, a mixed homogeneous polynomial of bidegree $(k,k)$, a conjugate partially symmetric coefficient tensor $\mathcal{A}$, and as a compressed formal inner product involving the normalized symmetric matricization of $\mathcal{A}$. Then, in Section \ref{$t$-Hermitian Forms}, we will introduce the central object of study: degree $k$ $t$-Hermitian forms. In particular, we establish their bijection with $t$-conjugate partially symmetric tubal tensors, their Fourier decomposition into classical degree $k$ Hermitian forms, and the spectral characterization of their positivity. After that, in Section \ref{Application to Classical Hermitian Forms} we consider classical Hermitian forms obtained from tensor contractions of tubal tensors corresponding to $t$-Hermitian forms. In particular, we characterize when such contractions preserve positivity, establish quantitative lower bounds for resulting contraction-induced forms, and then demonstrate these results with an explicit numerical example. After a brief conclusion summarizing the main results of this manuscript, we then provide a short appendix including proofs of a few auxiliary results used throughout the paper. 

\subsection{Notation/Terminology}
\begin{itemize}
    \item For a positive integer $n$, $[n] := \{1,2,\dots,n\}$. 
    \item $\mathcal{A}$ will denote a coefficient tensor, which we refer to as a \textbf{hypermatrix}. If $\mathcal{A}\in \mathbb{C}^{n_1\times\dots\times n_N}$, then we say that $\mathcal{A}$ has \textbf{order} $N$. When $n_1=n_2=\dots=n_N$, we say that $\mathcal{A}$ is \textbf{cubical}.  
    \item For any hypermatrix $\mathcal{A}\in \mathbb{C}^{n_1\times\dots\times n_N}$, $\mathcal{A}[i_1,\dots,i_N]$ denotes the $(i_1,\dots,i_N)$-coordinate of $\mathcal{A}$.
    \item $\langle\text{ , }\rangle$ denotes the Frobenius inner product, which for order $N$ hypermatrices $\mathcal{A},\mathcal{B}\in \mathbb{C}^{n_1\times\dots\times n_N}$ is defined as 
    \[\langle \mathcal{A},\mathcal{B}\rangle := \sum\limits_{i_1,\dots,i_N=1}^{n_1,\dots,n_N}\mathcal{A}[i_1,\dots,i_N]\mathcal{B}[i_1,\dots,i_N].\]
    \item $\|\cdot\|$ denotes the Frobenius norm, induced by the Frobenius inner product. 
    \item $\otimes$ denotes the formal tensor product, which in coordinates is realized by the Kronecker product. $\odot$ will either denote the formal symmetric tensor product: $u\odot v := \frac{1}{2}(u\otimes v+v\otimes u)$, or the matrix Hadamard product: if $A\in \mathbb{C}^{m\times n}$ and $B\in \mathbb{C}^{m\times n}$, then $A\odot B\in \mathbb{C}^{m\times n}$ such that $(A\odot B)[i,j] = A[i,j]B[i,j]$ for each $i\in [m]$ and $j\in [n]$; context will make it clear which operation $\odot$ denotes. 
    \item $\circledast$ denotes the discrete circular convolution operator: for $p$-dimensional vectors/tubes, 
    \[(a\circledast b)[\ell] = \sum\limits_{m=1}^pa[m]b\left[\big((\ell-m)\bmod p\big)+1\right]\] 
    for each $\ell\in [p]$. 
    \item $\mathrm{fft}$ denotes the Fast Fourier Transform, equivalent to the unnormalized Discrete Fourier Transform matrix; in particular, if $v$ is an $p$-dimensional vector, then $\mathrm{fft}(v) = \sqrt{p}\mathrm{DFT}_pv$, where recall $\mathrm{DFT}_p = \frac{1}{\sqrt{p}}\left[\omega_p^{(i-1)(j-1)}\right]_{p\times p}$ with $\omega_p := e^{-\frac{2\pi i}{p}}$. 
    \item $\circ$ denotes the Segre outer product: $\mathcal{A}\in \mathbb{C}^{m_1\times \dots m_M}$ and $\mathcal{B}\in \mathbb{C}^{n_1\times\dots\times n_N}$, then $\mathcal{A}\circ \mathcal{B}\in \mathbb{C}^{m_1\times\dots\times m_M\times n_1\times\dots\times n_N}$ such that $(\mathcal{A}\circ \mathcal{B})[i_1,\dots,i_M,j_1\dots,j_N] = \mathcal{A}[i_1,\dots,i_M]\mathcal{B}[j_1,\dots,j_N]$. 
\end{itemize}

\section{The Different Representations of Classical Hermitian Forms}\label{The Different Representations of Classical Hermitian Forms}
In this section we give a broad overview of arbitrary degree Hermitian forms and the various different ways we may be able to represent them, in particular in terms of both hypermatrices and certain types of matrices. First, we begin with a high-level definition of Hermitian forms as the diagonal of sesquilinear forms. Next, by considering their coefficient tensors, we show how we may view them as as a hypermatrix product, which naturally facilitates the study of higher degree positive Hermitian forms. After reviewing the basics of higher degree Hermitian positive forms, we then turn our attention to how we can represent such Hermitian forms as formal inner products in terms of matrices through tensor unfolding. After discussing the sub-optimal implications this point of view has when it comes to the study of higher degree Hermitian positive forms, we then introduce the notion symmetric matricization, which takes the higher order symmetries of coefficient tensors into account. In addition to allowing us to compress the formal inner product representation of higher degree Hermitian forms, symmetric matricization enables us to derive simple yet useful sufficient condition for strict Hermitian positivity of higher degree Hermitian forms. 

\subsection{High-Level Definition of Hermitian Forms}\label{High-Level Definition of Hermitian Forms}
Hermitian quadratic forms are typically defined as the diagonal of a two-argument sesquilinear form. Higher degree Hermitian forms admit an analogous description in terms of sesquilinear maps with two separately symmetric blocks of arguments, as consequence from the polarization correspondence between homogeneous polynomials and symmetric multilinear maps \cite{thomas2014polarization}. Below we make the relevant higher bi-degree structure explicit. 
\begin{defn}
    Let $V$ be a complex vector space with $\dim(V)=n$. A \textbf{$\boldsymbol{(k,k)}$-sesquilinear form on $V$} is a map 
    \[S:\underbrace{V\times\dots\times V}_{2k\text{ times}}\rightarrow \mathbb{C}\]
    which is conjugate-linear in each of its first $k$ arguments and linear in each of its last $k$ arguments. That is, for any $\alpha,\beta\in \mathbb{C}$ and $u_i,v_i,w_i\in V$, 
    \begin{align*}
        &S(u_1,\dots,\alpha u_i+\beta w_i,\dots,u_k;v_1,\dots,v_k) \\
        &\quad = \overline{\alpha}S(u_1,\dots,u_i,\dots,u_k;v_1,\dots,v_k)+\overline{\beta}S(u_1,\dots,w_i,\dots,u_k;v_1,\dots,v_k)
    \end{align*}
    and 
    \begin{align*}
        &S(u_1,\dots,u_k;v_1,\dots,\alpha v_i + \beta w_i,\dots,v_k) \\
        &\quad = \alpha S(u_1,\dots,u_k;v_1,\dots,v_i,\dots,v_k)+\beta S(u_1,\dots,u_k;v_1,\dots,w_i,\dots,v_k).
    \end{align*}
    We say that $S$ is \textbf{separately symmetric} if 
    \[S(u_{\sigma(1)},\dots,u_{\sigma(k)};v_{\tau(1)},\dots,v_{\tau(k)}) = S(u_1,\dots,u_k;v_1,\dots,v_k)\]
    for all $u_1,\dots,u_k,v_1,\dots,v)k\in V$ and every $\sigma,\tau\in S_k$. We say that $S$ is \textbf{Hermitian} if 
    \[S(u_1,\dots,u_k;v_1,\dots,v_k) = \overline{S(v_1,\dots,v_k;u_1,\dots,u_k)}\]
    for all $u_1,\dots,u_k,v_1,\dots,v_k\in V$. 
\end{defn}
\begin{remark}
    Separate symmetry allows $S$ to be regarded as an ordinary two-argument sesquilinear form 
    \[\widetilde{S}:\mathrm{Sym}^k(V)\times \mathrm{Sym}^k(V)\rightarrow \mathbb{C}\]
    satisfying 
    \[\widetilde{S}(u_1\odot\dots\odot u_k, v_1\odot\dots\odot v_k) = S(u_1,\dots,u_k;v_1,\dots,v_k),\]
    with $\odot$ denoting the symmetric tensor product. Moreover, $S$ is Hermitian if and only if $\widetilde{S}$ is Hermitian in the usual two-argument sense. Thus, a separately symmetric $(k,k)$-sesquilinear form on $V$ may equivalently be viewed as an ordinary Hermitian sesquilinear form on $\mathrm{Sym}^k(V)$. 
\end{remark}
\begin{defn}
    Let $S$ be a separately symmetric Hermitian $(k,k)$-sesquilinear form on $V$. The map 
    \begin{align*}
        h_S:V&\longrightarrow \mathbb{C} \\
        v&\mapsto S(\underbrace{v,\dots,v}_{k\text{ times}};\underbrace{v,\dots,v}_{k\text{ times}})
    \end{align*}
    is called the \textbf{degree $\boldsymbol{k}$ Hermitian form} associated with $S$. 
\end{defn}
\begin{remark}\label{Properties of Degree k Hermitian Forms}
    Equivalently, in terms of the induced Hermitian sesquilinear form $\widetilde{S}$ on $\mathrm{Sym}^k(V)$, 
    \[h_S(v) = \widetilde{S}(v^{\odot k},v^{\odot k}).\]
    Note also that the Hermitian property of $S$ implies that 
    \[\overline{h_S(v)} = h_S(v)\]
    for every $v\in V$, hence $h_S$ is a real-valued function. Moreover, by multi-sesquilinearity, 
    \[h_S(\lambda v) = |\lambda|^{2k}h_S(v)\]
    for all $\lambda\in \mathbb{C}$. Thus, $h_S$ is a homogenous function of total degree $2k$. 

    After fixing a basis $\{e_1,\dots,e_n\}$ for $V$, let $x_1,\dots,x_n\in V^*$ denote the coordinate functionals so that $x_i(v)$ is the $i^{th}$-coordinate of $v$ in the basis $\{e_1,\dots,e_n\}$, for each $i\in [n]$. and every $v\in V$. Then $v = \sum\limits_{i=1}^nx_i(v)e_i$ for every $v\in V$, and so in the basis $\{e_1,\dots,e_n\}$ $h_S$ may equivalently be viewed as a function of $x_1,\dots,x_n$. Furthermore, denoting
    \[a_{i_1\dots i_k j_1\dots j_k} := S(e_{i_1},\dots,e_{i_k};e_{j_1},\dots,e_{j_k})\]
    for each $i_1,\dots,i_k,j_1,\dots,j_k\in [n]$, by multi-sesquilinearity of $S$ it follows that 
    \begin{equation}\label{formal polynomial expression of Hermitian form}
        h_S(x_1,\dots,x_n) = \sum\limits_{\substack{i_1,\dots,i_k=1 \\ j_1,\dots,j_k=1}}^na_{i_1\dots i_k j_1\dots j_k}\overline{x_{i_1}}\dots \overline{x_{i_k}}x_{j_1}\dots x_{j_k},
    \end{equation}
    where for convenience we have omitted the input arguments for the $x_i$; that is, in a fixed basis, $h_S$ is a bihomogenous mixed polynomial of bi-degree $(k,k)$. Moreover, by separate symmetry and the Hermitian property of $S$, the coefficients in equation \eqref{formal polynomial expression of Hermitian form} satisfy 
    \[\overline{a_{j_{\sigma(1)}\dots j_{\sigma(k)}i_{\tau(1)}\dots i_{\tau(k)}}} = a_{i_1\dots i_k j_1\dots j_k}\]
    for all $i_1,\dots,i_k,j_1,\dots,j_k\in [n]$ and every $\sigma,\tau\in S_k$. Such polynomials are an instance of what has been referred to in the literature as Hermitian symmetric polynomials; see for example \cite{d2011hermitian}. Since the higher-degree setting is particularly important for our purposes, we shall refer to polynomials of the form in equation \eqref{formal polynomial expression of Hermitian form} as degree $k$ Hermitian forms. 
\end{remark}

\subsection{The Hypermatrix Representation of Hermitian Forms}
From Remark \ref{Properties of Degree k Hermitian Forms}, a degree $1$ Hermitian form (which has total degree $2$ and hence is commonly referred to as a Hermitian quadratic form) $h$ in the indeterminates $x_1,\dots,x_n$ is of the form: 
\begin{equation}\label{Hermitian forms - definition}
    h(x_1,\dots,x_n) = \sum\limits_{i,j=1}^nh_{ij}\overline{x}_ix_j.
\end{equation}
where 
\[H := [h_{ij}]\in \mathbb{C}^{n\times n}\] 
is a Hermitian matrix, and such that $ev_z(\overline{x}) = \overline{z}$ for any $z\in \mathbb{C}$. In addition to guaranteeing real output for when $x_1,...,x_n$ are evaluated at complex numbers, the Hermitian property of $H$ also guarantees that the correspondence $h\mapsto H$ is well-defined; in other words, for fixed choice of basis, the map 
\begin{align*}
    \{\text{Hermitian quadratic forms in }x_1,\dots,x_n\} &\longrightarrow \{n\times n\text{ Hermitian matrices}\} \\
    h(x_1,\dots,x_n) &\mapsto H
\end{align*}
is a bijection. Moreover, setting 
\[x := \left[\begin{array}{c}
    x_1 \\
    \vdots \\
    x_n
\end{array}\right],\]
we may we may rewrite equation \eqref{Hermitian forms - definition} as the matrix product
\begin{equation}\label{Hermitian forms - matrix expression}
    h(x_1,...,x_n) = x^{\dagger}Hx,
\end{equation}
which we refer to as the \textbf{matrix representation} of the Hermitian quadratic form $h$. The notion of matrix representation of degree $1$ Hermitian forms is naturally extended to degree $k$ Hermitian forms in terms of hypermatrices via multilinear matrix multiplication, which we define as follows:
\begin{defn}\label{multilinear-matrix-multiplication}
    Let $N\geq 2$ and for each $j\in [N]$, let $x^j = \left[\begin{array}{ccc}
        x^j_1 & \dots & x^j_{n_j}
    \end{array}\right]^T$ be an $n_j$-dimensional vector of complex variables. Suppose also that $\mathcal{A}\in \mathbb{C}^{n_1\times\dots\times n_N}$ is an order $N$ hypermatrix. The \textbf{(right) multilinear matrix multiplication\footnote{In general the multilinear matrix multiplication operation may be defined in terms of a tuple consisting of matrices or a combination of vectors in matrices, in which case the output may be a vector, matrix, or hypermatrix depending on the dimensions of each tuple input; in particular, when the tuple consists of invertible matrices, multilinear matrix multiplication reduces to the change of basis transformation for covariant tensors (see \cite{lim2013tensors} for the more general formulation of multilinear matrix multiplication and further discussion of its tensorial interpretation). However for our purposes, we will only care about the case when the tuple inputs are vectors, hence for the sake of simplicity we only define it in terms of a tuple with vector inputs.} of $\mathcal{A}$} with the tuple $(x^1,\dots,x^N)$ is given by the polynomial 
    \[\mathcal{A}*(x^1,\dots,x^N) := \sum\limits_{i_1,...,i_N}^{n_1\dots,n_N}\mathcal{A}_{i_1\dots i_N}x^1_{i_1}\cdot...\cdot x^N_{i_N}.\] 
\end{defn}
In particular, viewing order $2$ hypermatrices as matrices, if $A$ is an $m\times n$ matrix, $x$ is an $m$-dimensional vector of complex variables, and $y$ is an $n$-dimensional vector of complex variables, then  
\begin{equation}\label{Relationship between MM and MMM}
    A*(x,y) = x^TAy,
\end{equation} 
consequently it follows from equation \eqref{Hermitian forms - matrix expression} that every degree $1$ Hermitian form $h$ in the variables $x_1,\dots,x_n$ can be expressed as the formal multilinear matrix product
\begin{equation}\label{Hermitian quadratic form as MMM product}
    h(x_1,...,x_n) = H*(\overline{x},x)
\end{equation}
for some $n\times n$ Hermitian matrix $H$. Indeed, by Remark \ref{Properties of Degree k Hermitian Forms}, every degree $k$ Hermitian form $h$ in the variables $x_1,...,x_n$ may expressed as the formal multilinear matrix product:
\begin{equation}\label{hypermatrix representation}
    h(x_1,...,x_n) = \mathcal{A}*(\underbrace{\overline{x},...,\overline{x}}_{k\text{ times}},\overbrace{x,...,x}^{k\text{ times}})
\end{equation}
for some order $2k$ cubical hypermatrix $\mathcal{A} = [a_{i_1...i_kj_1...j_k}]\in \mathbb{C}^{n\times...\times n}$ such that 
\[\overline{a_{j_{\sigma(1)}...j_{\sigma(k)}i_{\tau(1)}...i_{\tau(k)}}} = a_{i_1\dots i_k j_1\dots j_k}\] for each $i_1,...,i_k,j_1,...,j_k\in [n]$ and every $\sigma,\tau\in S_k$. Such hypermatrices $\mathcal{A}$ are commonly referred to as \textbf{conjugate partially symmetric} \cite{jiang2016characterizing}, and we refer to equation \eqref{hypermatrix representation} as the \textbf{hypermatrix representation} of a degree $k$ Hermitian form $h$. It then follows that there is a bijection between degree $k$ Hermitian forms and order $2k$ Hermitian partially symmetric hypermatrices (indeed this was explicitly proven in \cite{jiang2016characterizing}). We formally summarize with the following fact.
\begin{fact}\label{Hermitian Form Bijection}\cite[Proposition 3.8]{jiang2016characterizing}
    Let $h$ be a degree $k$ Hermitian form in the complex variables $x_1,\dots,x_n$, and let $\mathcal{A} = [a_{i_1\dots i_k j_1\dots j_k}]\in \mathbb{C}^{n\times\dots\times n}$ be an order $2k$ \textbf{conjugate partially symmetric} hypermatrix, meaning that 
    \[\overline{a_{j_{\sigma(1)}\dots j_{\sigma(k)}i_{\tau(1)}\dots i_{\tau(k)}}} = a_{i_1\dots i_k j_1\dots j_k}\]
    for all $i_1,\dots,i_k,j_1,\dots,j_k\in [n]$ and every $\sigma,\tau\in S_k$. Then there is a bijection between order $2k$ conjugate partially symmetric hypermatrices $\mathcal{A}\in \mathbb{C}^{n\times\dots\times n}$ and degree $k$ Hermitian forms $h$ in the complex variables $x_1,\dots,x_n$, and in particular it is given by the map 
    \[\mathcal{A}\mapsto h_{\mathcal{A}}(x) := \mathcal{A}*(\overline{x},\dots,\overline{x},x,\dots,x)\]
    (with $x := \left[\begin{array}{ccc}
        x_1 & \dots & x_n
    \end{array}\right]^T$). 
\end{fact}
Thus, we may unambiguously associate any degree $k$ Hermitian form with a unique order $2k$ Hermitian partially symmetric hypermatrix, and vice versa. 

\subsubsection{Positive Definite Hermitian Forms}
The hypermatrix representation of Hermitian forms naturally facilitates the study higher degree positive Hermitian forms. For a more in depth treatment of the theory, see \cite{chen2025hat}. 
\begin{defn}\label{Hermitian positive definite form - classical}
    A degree $k$ Hermitian form $h$ in the complex variables $x := \left[\begin{array}{ccc}
        x_1 & \dots & x_n 
    \end{array}\right]^T$ is \textbf{Hermitian positive definite (semidefinite)} if $h$ is positive under evaluation; that is, if 
    \[h\big(\mathrm{Ev}_z(x)\big) > 0\text{ $(\geq 0)$}\]
    for every nonzero $z\in \mathbb{C}^n$, with $\mathrm{Ev}_z(x) = \left[\begin{array}{c}
        z_1 \\
        \vdots \\
        z_n
    \end{array}\right]$ and $\mathrm{Ev}_z(\overline{x}) = \left[\begin{array}{c}
        \overline{z_1} \\
        \vdots \\
        \overline{z_n}
    \end{array}\right] =: \overline{z}$. 
\end{defn}
\begin{remark}
    Equivalently, if $\mathcal{A}$ denotes the unique conjugate partially symmetric hypermatrix representation of $h$, then $h$ is Hermitian positive definite (semidefinite) if and only if 
    \begin{equation}\label{Heermitian positive definite tensor}
        \mathcal{A}*(\underbrace{\overline{z},\dots,\overline{z}}_{k\text{ times}},\overbrace{z,\dots,z}^{k\text{ times}}) > 0\text{ $(\geq 0)$}
    \end{equation}
    for all nonzero $z\in \mathbb{C}^n$. 
    For this reason, we say that the hypermatrix $\mathcal{A}$ itself is \textbf{Hermitian positive definite (semidefinite)} if \eqref{Heermitian positive definite tensor} holds for all nonzero $z\in \mathbb{C}^n$. 
\end{remark}
The theory of tensor eigenvalues was developed independently by both L.H. Lim and L. Qi in \cite[2005]{lim2005singular} and \cite[2005]{qi2005eigenvalues} respectively, however in both of their work the base field is $\mathbb{R}$. In \cite[2007]{ni2007degree}, G. Ni, L. Qi, F. Wang, and Y. Wang introduce the notion of $E$-eigenvalues of tensors, where the base field is $\mathbb{C}$, and in \cite{zhang2012quantum}, X. Zhang and L. Qi introduces the notion of $Q$-eigenvalues, defined in terms of complex variables. Then, B. Jiang, Z. Li, and S. Zhang, as well as H. Chen and Y. Yang, continue this study of tensor eigenvalues, defining the following classes of tensor eigenvalues: 
\begin{defn} 
    Let $\mathcal{A} = [a_{i_1...i_kj_1...j_k}]\in \mathbb{C}^{n\times...\times n}$ be an order $2k$ cubical hypermatrix. In \cite{jiang2016characterizing}, Jiang et al. say that $\lambda\in \mathbb{C}$ is an \textbf{$\boldsymbol{\widehat{C}}$-eigenvalue} of $\mathcal{A}$ if there exists a vector $z = [z_1,...,z_n]^T\in \mathbb{C}^n$ such that 
    \[z^{\dagger}z = 1\]
    and 
    \[\mathcal{A}*(I_n,\underbrace{\overline{z},...,\overline{z}}_{k-1\text{ times}},\overbrace{z,...,z}^{k\text{ times}}) = \lambda z.\]
    
    In \cite{chen2025hat}, Chen \& Yang consider a slightly more general notion than that of $\widehat{C}$-eigenvalues, and say that  $\lambda\in \mathbb{C}$ is an \textbf{$\boldsymbol{\widehat{H}}$-eigenvalue} of $\mathcal{A}$ if there exists a vector $z = [z_1,...,z_n]^T\in \mathbb{C}^n$ such that 
    \[\sum\limits_{i_2,...,i_k,j_1,...,j_k=1}^na_{ii_2...i_kj_1...j_k}\overline{z}_{i_2}...\overline{z}_{i_k}z_{j_1}...z_{j_k} = \lambda |z_i|^{2(k-1)}z_i\quad \forall i=1,...,n,\]
    or equivalently
    \[\mathcal{A}*(I_n,\underbrace{\overline{z},...,\overline{z}}_{k-1\text{ times}},\overbrace{z,...,z}^{k\text{ times}}) = \lambda |z| z.\]
\end{defn}
Chen \& Yang then prove the following result, providing a spectral characterization of Hermitian positive definite hypermatrices. 
\begin{theorem}[Spectral Theorem of Hermitian Forms]\cite[Theorem 1.4]{chen2025hat}\label{Spectral Theorem of Hermitian Forms}
    Let $\mathcal{A}\in \mathbb{C}^{n\times...\times n}$ be an order $2k$ conjugate partially symmetric hypermatrix. Then there always exists $\widehat{H}$-eigenvalues (and hence $\widehat{C}$-eigenvalues) of $\mathcal{A}$. Moreover, $\mathcal{A}$ is Hermitian positive definite (semi-definite) if and only if all of the $\widehat{H}$-eigenvalues of $\mathcal{A}$ are positive (nonnegative). 
\end{theorem}
In the next main section of this article, we will generalize this theorem to certain types of order $2k+1$ hypermatrices in terms of the $t$-product. 

\subsection{The Matrix Representations of Higher Degree Hermitian Forms}
Since a degree $k$ Hermitian form in the complex variables $x_1,\dots,x_n$ may be uniquely represented by an order $2k$ conjugate partially symmetric hypermatrix $\mathcal{A}\in \mathbb{C}^{n\times\dots\times n}$, which itself consists of $n^{2k}$ entries, by unfolding this hypermatrix into an $n^k\times n^k$ matrix we should also be able to represent said Hermitian form as a matrix. To represent arbitrary degree $k$ Hermitian forms as vector-matrix-vector products, we will need a well-defined protocol for unfolding tensors. For a general overview of tensor unfolding, see S. Y. Chang's article \cite{chang2023tensor}. Note, however, that we present a slightly modified protocol for tensor unfolding (so that later formulas in the appendix are easier to derive), which we define below.
\begin{defn}[Tensor Unfolding]\label{Matricization}
    Let $\mathcal{A} = [a_{i_1\dots i_M j_1\dots j_N}]\in \mathbb{C}^{m_1\times...\times m_M\times n_1\times...\times n_N}$, $\mathbf{m} = (m_1,...,m_M)$ and $\mathbf{n} = (n_1,...,n_N)$, $|\mathbf{m}| := m_1\cdot...\cdot m_M$, $|\mathbf{n}| := n_1\cdot...\cdot n_N$, and lastly let $\mathbf{i} = (i_1,...,i_M)\in [m_1]\times...\times [m_M]$ and $\mathbf{j} = (j_1,...,j_N)\in [n_1]\times...\times [n_N]$. Then the \textbf{$\boldsymbol{(m,n)}$-matricization} of $\mathcal{A}$ is the $|\mathbf{m}|\times |\mathbf{n}|$ matrix, denoted $M_{\mathbf{m}\times \mathbf{n}}(\mathcal{A})$, such that 
    \[\big(M_{\mathbf{m}\times \mathbf{n}}(\mathcal{A})\big)_{\psi(\mathbf{i},\mathbf{m}),\psi(\mathbf{j},\mathbf{n})} = a_{\mathbf{i}\mathbf{j}},\]
    where  
    \[\psi(\mathbf{i},\mathbf{m}) := 1+\sum\limits_{r=1}^M(i_r-1)\prod\limits_{s=1}^{M-r}m_{M-s}\quad\text{and}\quad \psi(\mathbf{j},\mathbf{n}) := 1+\sum\limits_{r=1}^N(j_r-1)\prod\limits_{s=1}^{N-r}n_{N-s}.\]
    In the special case when $M=N=k$ and $m_i=n_i=n$ for each $i\in [k]$, we define the \textbf{cubically balanced matricization} of $\mathcal{A}$, denoted $M_{cb}(\mathcal{A})$, to be the $n^k\times n^k$ matrix given by the $(\mathbf{n},\mathbf{n})$-matricization of $\mathcal{A}$. 
\end{defn}
\begin{remark}
    In particular, for an order $2k$ cubical hypermatrix $\mathcal{A}\in \mathbb{C}^{n\times\dots\times n}$, the cubically balanced matricization of $\mathcal{A}$ is given by 
    \[\big(M_{cb}(\mathcal{A})\big)_{\psi(\mathbf{i}),\psi()\mathbf{j}} = a_{\mathbf{i}\mathbf{j}},\]
    with $\psi(\mathbf{i}) := 1+n^{k-\ell}\sum\limits_{\ell=1}^k(i_{\ell}-1)$ for any $\mathbf{i}\in [n]^{\times k}$. In particular, we omit the second argument of $\psi$ because in this special case it is not needed to prevent ambiguity. 
\end{remark}

\begin{proposition}\label{Matricized Hermitian Form}
    Let $\mathcal{A}\in \mathbb{C}^{n\times\dots\times n}$ be an order $2k$ cubical hypermatrix and $x = \left[\begin{array}{ccc}
        x_1 & \dots & x_n
    \end{array}\right]$ be a vector of complex variables. Then  
    \[\langle x^{\otimes k},M_{cb}(\mathcal{A})x^{\otimes k}\rangle = \mathcal{A}*(\overline{x},\dots,\overline{x},x,\dots,x),\]
    with "$\otimes$" denoting the Kronecker product and $\langle\text{ , }\rangle$ denoting the formal Frobenius inner product. Consequently, for any degree $k$ Hermitian form $h$ in the complex variables $x_1,\dots,x_n$, there is a unique order $2k$ conjugate partially symmetric hypermatrix $\mathcal{A}\in \mathbb{C}^{n\times\dots\times n}$ such that 
    \[h(x_1,\dots,x_n) = \langle x^{\otimes k},M_{cb}(\mathcal{A})x^{\otimes k}\rangle.\]
\end{proposition}
\begin{proof}
The first equality follows from the straightforward calculation:
\begin{align*}
    \langle x^{\otimes k},M_{cb}(\mathcal{A})x^{\otimes k}\rangle &= \sum\limits_{\substack{\mathbf{i}\in [n]^{\times k} \\
        \mathbf{j}\in [n]^{\times k}}}\big((x^{\otimes k})^{\dagger}\big)_{\psi(\mathbf{i})}\big(M_{cb}(\mathcal{A})\big)_{\psi(\mathbf{i}),\psi(\mathbf{j})}(x^{\otimes k})_{\psi(\mathbf{j})} \\
    &= \sum\limits_{\substack{\mathbf{i}\in [n]^{\times k} \\ \mathbf{j}\in [n]^{\times k}}}a_{\mathbf{i}\mathbf{j}}(\overline{x}^{\otimes k})_{\psi(\mathbf{i})}(x^{\otimes k})_{\psi(\mathbf{j})} \\
    &= \sum\limits_{\substack{i_1,...,i_k=1 \\ j_1,...,j_k=1}}^na_{i_1...i_kj_1...j_k}\overline{x}_{i_1}\dots\overline{z}_{i_k}x_{j_1}...x_{j_k} \\
    &= \mathcal{A}*(\overline{x},\dots,\overline{x},x,\dots,x).
\end{align*}
Consequently, if $\mathcal{A}$ is assumed to be conjugate partially symmetric, then $\mathcal{A}*(\overline{x},\dots,\overline{x},x,\dots,x) = \langle x^{\otimes k},M_{cb}(\mathcal{A})x^{\otimes k}\rangle$ uniquely corresponds to some degree $k$ Hermitian form $h(x_1,\dots,x_n)$. 
\end{proof}
With this proposition, we obtain an immediate sufficient condition for Hermitian positive semidefiniteness. 
\begin{fact}\label{Sufficient Condition for Hermitian Positive Semidefiniteness}
    Let $\mathcal{A}\in \mathbb{C}^{n\times\dots\times n}$ be an order $2k$ conjugate partially symmetric hypermatrix. If $M_{cb}(\mathcal{A})$ is Hermitian positive semidefinite, then $\mathcal{A}$ is Hermitian positive semidefinite. 
\end{fact}
A few remarks are in order. 
\begin{remark}
    First, note that the converse does not hold. For example, consider the diagonal hypermatrix $\mathcal{I}\in \mathbb{C}^{2\times 2\times 2\times 2}$ with a $1$ in its $(1,1,1,1)$ and $(2,2,2,2)$ entries, and $0$ for all its other entries (that is, $\mathcal{I}$ is the order $4$ analogue of the identity matrix), and consider the hypermatrix $\mathcal{J}\in \mathbb{C}^{2\times 2\times 2\times 2}$ with a $1$ in its $(1,2,1,2)$, $(1,2,2,1)$, $(2,1,2,1)$, and $(2,1,1,2)$ entries, and $0$ for all its other entries (that is, $\mathcal{J}$ is the order $4$ analogue of the $2\times 2$ exchange matrix). One can check that $\mathcal{I}-\frac{1}{4}\mathcal{J}$ is conjugate partially symmetric, and furthermore that 
    \[h_{\mathcal{I}-\frac{1}{4}\mathcal{J}}(x_1,x_2) = \left(\mathcal{I}-\frac{1}{4}\mathcal{J}\right)*(\overline{x},\overline{x},x,x) = |x_1|^4+|x_2|^4-|x_1|^2|x_2|^2.\]
    Setting $a := |x_1|^2$ and $b := |x_2|^2$, it follows that $|x_1|^4+|x_2|^4-|x_1|^2|x_2|^2 = (a-b)^2+ab>0$ for all $(a,b)\neq (0,0)$, hence $\mathcal{I}-\frac{1}{4}\mathcal{J}$ is Hermitian positive definite (and thus trivially Hermitian positive semidefinite). However, the cubically balanced matricization of $\mathcal{I}-\frac{1}{4}\mathcal{J}$ is given by 
    \[M_{cb}\left(\mathcal{I}-\frac{1}{4}\mathcal{J}\right) = \left[\begin{array}{cccc}
        1 & 0 & 0 & 0 \\
        0 & -\frac{1}{4} & -\frac{1}{4} & 0 \\
        0 & -\frac{1}{4} & -\frac{1}{4} & 0 \\
        0 & 0 & 0 & 1
    \end{array}\right]\]
    which has an eigenvalues $1$, $1$, $0$, and $-\frac{1}{2}$. Thus, $M_{cb}\left(\mathcal{I}-\frac{1}{4}\mathcal{J}\right)$ is not Hermitian positive semidefinite. The deeper reason for this phenomenon is that in general tensor eigenvalues do not coincide with the matrix eigenvalues of corresponding tensor unfoldings. Thus, while the tensor eigenvalues of $\mathcal{I}-\frac{1}{4}\mathcal{J}$ are all strictly positive by the Spectral Theorem of Hermitian Forms, the matrix eigenvalues of its cubically balanced matricization are not all positive (and can in fact be negative). 
\end{remark} 
\begin{remark}
    Another salient point is that Fact \ref{Sufficient Condition for Hermitian Positive Semidefiniteness} can only be stated in terms of Hermitian positive semidefiniteness because in general the cubically balanced matricization of a conjugate partially symmetric hypermatrix will have non-trivial kernel. As observed above, $M_{cb}\left(\mathcal{I}-\frac{1}{4}\mathcal{J}\right)$ has a $0$ eigenvalue, and in general this is unavoidable due to the higher order symmetries of conjugate partially symmetric hypermatrices (atleast when the order is greater than $2$, for otherwise we are dealing with a regular Hermitian matrix). We formalize this observation below with the following proposition. 
\end{remark}
\begin{proposition}\label{Cubically Balanced Matricizations Have Non-trivial Kernel}
    Let $\mathcal{A}\in\mathbb{C}^{n\times\dots\times n}$ be an order $2k$ conjugate partially symmetric hypermatrix with $k\geq 2$. Then $\ker\big({M_{cb}(\mathcal{A})}\big)\supseteq \mathrm{Sym}^k(\mathbb{C}^n)^{\perp}\neq \{0\}$. 
\end{proposition}
\begin{proof}
    Let $\mathbf{i} := (i_1,\dots,i_k)\in [n]^{\times k}$ and $\mathbf{j} := (j_1,\dots,j_k)\in [n]^k$. conjugate partial symmetry of $\mathcal{A}$ implies that 
    \begin{equation}\label{application of Hermitian partial symmetry}
        a_{\mathbf{i},\mathbf{j}} = \overline{a_{\mathbf{j},\sigma\cdot \mathbf{i}}}\qquad \forall \sigma\in S_k,
    \end{equation}
    where $\sigma\cdot \mathbf{i} := (i_{\sigma^{-1}(1)},\dots,i_{\sigma^{-1}(k)})$, hence again by conjugate partial symmetry we have that 
    \[a_{\mathbf{i}\mathbf{j}} = a_{\sigma\cdot \mathbf{i},\mathbf{j}}\qquad \forall \sigma\in S_k.\]
    Therefore,
    \[\big(M_{cb}(\mathcal{A})\big)_{\mathbf{i},\mathbf{j}} = \big(M_{cb}(\mathcal{A})\big)_{\psi(\mathbf{j}),\psi(\sigma\cdot \mathbf{i})},\]
    implying that every column of $M_{cb}$ is a vector in $\mathrm{Sym}^k(\mathbb{C}^{n})$, hence the column space of $M_{cb}$ is contained in $\mathrm{Sym}^k(\mathbb{C}^{n})$. Therefore, because $M_{cb}(\mathcal{A})$ is Hermitian (this is immediately implied by equation \eqref{application of Hermitian partial symmetry} by setting $\sigma$ to be the identity permutation), for every $v\in (\mathbb{C}^n)^{\otimes k}$ and $w\in \mathrm{Sym}^k(\mathbb{C}^n)^{\perp}$, we have 
    \[\langle v,M_{cb}(\mathcal{A})w\rangle = \langle \overbrace{M_{cb}(\mathcal{A})v}^{\mathrm{Sym}^k(\mathbb{C}^n)},w\rangle = 0.\]
    Thus, $\ker\big(M_{cb}(\mathcal{A})\big) \supseteq \mathrm{Sym}^k(\mathbb{C}^n)^{\perp}$, which is a non-trivial subspace of dimension $n^k - \binom{n+k-1}{k} \geq 1$ whenever $k\geq 2$. 
\end{proof}
Thus, in order to guarantee positivity of higher degree Hermitian forms we will need to take this partially symmetry condition into account in our tensor unfolding protocol. This naturally leads us to the notion of (normalized) symmetric matricizations, which we formally define in the next subsection. 

\subsubsection{Symmetric Matricizations}
For each $k\geq 2$, let $\mathcal{I}_k := \{(i_1,\dots,i_k)\in [n]^{\times k}:1\leq i_1\leq\dots\leq i_k\leq n\}$ be ordered lexicographically, from least to greatest; that is, the first element of $\mathcal{I}_k$ is $(1,\dots,1,1)$, the second is $(1,1,\dots,1,2)$, the $n^{th}$ element is $(1,1,\dots,n)$, the $(n+1)^{th}$ element is $(1,\dots,2,2)$, and so on with the last element being $(n,n,\dots,n)$. 
\begin{fact}\label{Position in lexicographic order}
    The position of the tuple $\mathbf{i} = (i_1,\dots,i_k)\in \mathcal{I}_k$ is given by 
    \[\varphi(\mathbf{i}) = 1+\sum\limits_{q=1}^k\binom{i_q+q-2}{q}.\]
\end{fact}
See Appendix Section A for verification of the above fact. 
\begin{defn}
    Let $\mathbf{i}\in \mathcal{I}_k$ and denote $\eta(\mathbf{i}) := |\mathrm{Orb}_{S_k}(\mathbf{i})|$ to be the cardinality of the orbit of $\mathbf{i}$; or in other words, $\eta(\mathbf{i})$ is the number of distinct permutations of $\mathbf{i}$. Furthermore, let $\mathcal{A}\in \mathbb{C}^{n\times\dots\times n}$ be an order $2k$ Hermitian partially symmetric hypermatrix. We define the \textbf{(normalized) symmetric matricization} of $\mathcal{A}$, denoted $M_{sym}(\mathcal{A})$, to be the $\binom{n+k-1}{k}\times \binom{n+k-1}{k}$ matrix such that 
    \[\big(M_{sym}\big)_{\varphi(\mathbf{i}),\varphi(\mathbf{j})} = \sqrt{\eta(\mathbf{i})\eta(\mathbf{j})}a_{\mathbf{i}\mathbf{j}}\qquad \forall \mathbf{i},\mathbf{j}\in\mathcal{I}_k.\]
\end{defn}
\begin{fact}\label{Symmetric Matricization Embedding}
    Let $\mathcal{A}\in \mathbb{C}^{n\times\dots\times n}$ be an order $2k$ conjugate partially symmetric hypermatrix. Then $M_{sym}(\mathcal{A})$ satisfies
    \[UM_{sym}(\mathcal{A})U^T = M_{cb}(\mathcal{A}),\]
    where $U$ is the real $n^k\times \binom{n+k-1}{k}$ matrix representing the isometric embedding $\mathrm{Sym}^k(\mathbb{C}^n)\hookrightarrow (\mathbb{C}^{n})^{\otimes k}$ given by 
    \[e_{\varphi(\mathbf{i})}\mapsto \frac{\sqrt{\eta(\mathbf{i})}}{k!}\sum\limits_{\sigma\in S_k}e_{\sigma(i_1)}\otimes\dots\otimes e_{\sigma(i_k)}.\]
    Consequently, we also have that 
    \[U^TM_{cb}(\mathcal{A})U = M_{sym}(\mathcal{A}),\]
    from which it follows that $M_{sym}(\mathcal{A})$ is a Hermitian matrix.  
\end{fact}
See Appendix Section A for verification of the above fact. 
\begin{example}
    Let $\mathcal{A}\in \mathbb{C}^{2\times 2\times 2\times 2}$ be the conjugate partially symmetric hypermatrix given by 
    \begin{align*}
        a_{1111} = 1 & & a_{1212} = 2 & & a_{2222} = 3 \\
        a_{1112} = 1+i & & a_{1122} = 2+i & & a_{1222} = 1+2i.
    \end{align*}
    The cubically balanced matricization of $\mathcal{A}$ is given by 
    \[M_{cb}(\mathcal{A}) = \left[\begin{array}{cccc}
        1 & 1+i & 1+i & 2+i \\
        1-i & 2 & 2 & 1+2i \\
        1-i & 2 & 2 & 1+2i \\
        2-i & 1-2i & 1-2i & 3
    \end{array}\right].\]
    The isometry $U$ is given by 
    \[U = \left[\begin{array}{ccc}
        1 & 0 & 0 \\
        0 & \frac{1}{\sqrt{2}} & 0 \\
        0 & \frac{1}{\sqrt{2}} & 0 \\
        0 & 0 & 1
    \end{array}\right]\]
    and it follows (either by definition or direct calculation) that 
    \[M_{sym}(\mathcal{A}) = \left[\begin{array}{ccc}
        1 & \sqrt{2}(1+i) & 2+i \\
        \sqrt{2}(1-i) & 4 & \sqrt{2}(1+2i) \\
        2-i & \sqrt{2}(1-2i) & 3
    \end{array}\right].\]
\end{example}
In general, symmetric matricizations of conjugate partially symmetric hypermatrices can have trivial kernel and in fact be Hermitian positive definite; below, we characterize precisely when they are guaranteed to be Hermitian positive definite. 
\begin{theorem}[Characterizing Positive Definiteness of Symmetric Matricizations]\label{Characterizing Positive Definiteness of Symmetric Matricizations}
    Let $\mathcal{A}\in \mathbb{C}^{n\times\dots\times n}$ be an order $2k$ conjugate partially symmetric hypermatrix. Then $M_{sym}(\mathcal{A})$ is Hermitian positive semidefinite if and only if $M_{cb}(\mathcal{A})$ is Hermitian positive semidefinite. Moreover, the following are equivalent:
    \begin{itemize}
        \item $M_{sym}(\mathcal{A})$ is Hermitian positive definite;
        \item $M_{cb}(\mathcal{A})$ is Hermitian positive semidefinite and $\ker\big(M_{cb}(\mathcal{A})\big) \cong \mathrm{Sym}^k(\mathbb{C}^n)^{\perp}$;
        \item $M_{cb}(\mathcal{A})$ is Hermitian positive semidefinite and $\mathrm{rank}\big(M_{cb}(\mathcal{A})\big) = \binom{n+k-1}{k}$. 
    \end{itemize}
\end{theorem}
\begin{proof}
    Since the matrix $U$ from Fact \ref{Symmetric Matricization Embedding} represents an isometric embedding $\mathrm{Sym}^k(\mathbb{C}^n)\hookrightarrow(\mathbb{C}^{n})^{\otimes k}$, every vector $z\in \mathbb{C}^{n^k}\cong (\mathbb{C}^n)^{\otimes k}$ can be uniquely written as 
    \[z = Uv+w\]
    with $v\in \mathrm{Sym}^k(\mathbb{C}^n)$ and $w\in \mathrm{Sym}^k(\mathbb{C}^n)^{\perp}$. Since $M_{cb}(\mathcal{A})$ annihilates $\mathrm{Sym}^k(\mathbb{C}^n)^{\perp}$ by Proposition \ref{Cubically Balanced Matricizations Have Non-trivial Kernel}, we have that 
    \begin{equation}\label{Cubically Balanced and Symmetric Matricization Positive Semidefiniteness}
        z^{\dagger}M_{cb}(\mathcal{A})z = v^{\dagger}U^TM_{cb}(\mathcal{A})Uv = v^{\dagger}M_{sym}(\mathcal{A})v,
    \end{equation}
    with the last equality following from Fact \ref{Symmetric Matricization Embedding}. Thus, we immediately have that $M_{sym}(\mathcal{A})$ is Hermitian positive semidefinite if and only if $M_{cb}(\mathcal{A})$ is Hermitian positive semidefinite. Moreover, if $M_{sym}(\mathcal{A})$ is Hermitian positive definite, then equation \eqref{Cubically Balanced and Symmetric Matricization Positive Semidefiniteness} implies that $M_{cb}(\mathcal{A})$ is Hermitian positive semidefinite and that the image of the isometry represented by $U$ has trivial intersection with the kernel of $M_{cb}(\mathcal{A})$; hence, $\ker\big(M_{cb}(\mathcal{A})\big) = \mathrm{Sym}^k(\mathbb{C}^n)^{\perp}$. Conversely, suppose $M_{cb}(\mathcal{A})$ is Hermitian positive semidefinite and $\ker\big(M_{cb}(\mathcal{A})\big) \cong \mathrm{Sym}^k(\mathbb{C}^n)^{\perp}$. Then $M_{sym}(\mathcal{A})$ is still Hermitian positive semidefinite, and furthermore if for any $v\in \mathbb{C}^{\binom{n+k-1}{k}}\cong \mathrm{Sym}^k(\mathbb{C}^n)$ we have that 
    \[v^{\dagger}M_{sym}(\mathcal{A})v = 0,\]
    then by Fact \ref{Symmetric Matricization Embedding} it follows that 
    \[0 = v^{\dagger}\big(U^TM_{cb}(\mathcal{A})U\big)v = (Uv)^{\dagger}M_{cb}(\mathcal{A})(Uv),\]
    which implies that $Uv\in \ker\big(M_{cb}(\mathcal{A})\big)\cong \mathrm{Sym}^k(\mathbb{C}^n)^{\perp}$. But $Uv\in \mathrm{Sym}^k(\mathbb{C}^n)$ too, hence $Uv = 0$, and so consequently $v=0$ since $U$ represents an isometry. Thus, $M_{sym}(\mathcal{A})$ must be strictly Hermitian positive definite. Lastly, we note that if $\ker\big(M_{cb}(\mathcal{A})\big) \cong \mathrm{Sym}^k(\mathbb{C}^n)^{\perp}$, then by the rank-nullity theorem, $\mathrm{rank}\big(M_{cb}(\mathcal{A})\big) = \binom{n+k-1}{k}$; conversely, if $\mathrm{rank}\big(M_{cb}(\mathcal{A})\big) = \binom{n+k-1}{k}$, then by the rank-nullity theorem again $\dim\Big(\ker\big(M_{cb}(\mathcal{A})\big)\Big) = n^k-\binom{n+k-1}{k}$, from which it follows that $\ker\big(M_{cb}(\mathcal{A})\big) \cong \mathrm{Sym}^k(\mathbb{C}^n)^{\perp}$. 
\end{proof}
Thus, with symmetric matricizations, we may restate Fact \ref{Sufficient Condition for Hermitian Positive Semidefiniteness} in terms of strict Hermitian positive definiteness after replacing $M_{cb}$ with $M_{sym}$. 
\begin{corollary}\label{Sufficient Condition for Hermitian Positive Definiteness}
    Let $\mathcal{A}\in \mathbb{C}^{n\times\dots\times n}$ be an order $2k$ conjugate partially symmetric hypermatrix and $x = \left[\begin{array}{ccc}
        x_1 & \dots & x_n
    \end{array}\right]$ be a vector complex variables. Then 
    \[\langle U^Tx^{\otimes k}, M_{sym}(\mathcal{A})U^Tx^{\otimes k} \rangle = \mathcal{A}*(\overline{x},\dots,\overline{x},x\dots,x),\]
    again with $"\otimes$" denoting the Kronecker product and $\langle\text{ , }\rangle$ denoting the formal Frobenius inner product. Consequently, if $M_{sym}(\mathcal{A})$ is Hermitian positive definite, then $\mathcal{A}$ is Hermitian positive definite. 
\end{corollary}
\begin{proof}
    Since $U$ represents an isometry,
    \[\langle U^Tx^{\otimes k}, M_{sym}(\mathcal{A})U^Tx^{\otimes k} \rangle = \langle x^{\otimes k},UM_{sym}(\mathcal{A})U^Tx^{\otimes k}\rangle = \langle x^{\otimes k},M_{cb}(\mathcal{A})x^{\otimes k}\rangle,\]
    which by Fact \ref{Matricized Hermitian Form} is precisely equal to $\mathcal{A}*(\overline{x},\dots,\overline{x},x,\dots,x)$. It then immediately follows by definition that if $M_{sym}(\mathcal{A})$ is Hermitian positive definite, then $\mathcal{A}$ is Hermitian positive definite. 
\end{proof}
\begin{remark}
    We quickly note that, similarly as with Fact \ref{Sufficient Condition for Hermitian Positive Semidefiniteness}, the converse of Corollary \ref{Sufficient Condition for Hermitian Positive Definiteness} also does not hold in general. For example, consider again the Hermitian positive definite hypermatrix $\mathcal{I}-\frac{1}{4}\mathcal{J}$. It's symmetric matricization is given by 
    \[M_{sym}\left(\mathcal{I}-\frac{1}{4}\mathcal{J}\right) = \left[\begin{array}{ccc}
        1 & 0 & 0 \\
        0 & -\frac{1}{2} & 0 \\
        0 & 0 & 1
    \end{array}\right],\]
    which has eigenvalues $1$, $1$, and $-\frac{1}{2}$, hence $M_{sym}\left(\mathcal{I}-\frac{1}{4}\mathcal{J}\right)$ is not Hermitian positive definite. 
\end{remark}
With this broad overview of higher degree Hermitian forms, their various representations, and sufficient conditions for guaranteeing their positivity/non-negativity, we may now introduce the main object of study in this article: $t$-Hermitian forms, which are both a generalization of classical Hermitian forms, and in particular may be viewed as a lift of collections of classical Hermitian forms via the $t$-product extended to higher order tensors. 

\section{$t$-Hermitian Forms}\label{$t$-Hermitian Forms}
Just as classical Hermitian forms may be uniquely represented as hypermatrices via multilinear matrix multiplication, $t$-Hermitian forms may be uniquely represented as higher-order tubal hypermatrices via a synthesis of multilinear matrix multiplication with the $t$-product. Therefore, before explicitly defining $t$-Hermitian forms, we first give a very brief overview of the $t$-product. 
\subsection{Review of the $t$-Product on Order $3$ Tensors}
\begin{defn}
    For an order 3 hypermatrix $\mathcal{A} = [a_{ijk}]\in \mathbb{C}^{m\times n\times p}$, the \textbf{$\boldsymbol{k^{th}}$ frontal slice} of $\mathcal{A}$ as the matrix 
    \[\mathcal{A}_{(k)} := \mathcal{A}[:,:,k] = [a_{ijk}]_{m\times n\times 1}\in \mathbb{C}^{m\times n}\]
    (that is, the index $k$ is fixed). The \textbf{block-circulant matrix} of $\mathcal{A}$ is then defined to be unique matrix 
    \[\mathrm{bcirc}(\mathcal{A}) := \left[\begin{array}{cccc}
        \mathcal{A}_{(1)} & \mathcal{A}_{(p)} & \dots & \mathcal{A}_{(2)} \\
        \mathcal{A}_{(2)} & \mathcal{A}_{(1)} & \dots & \mathcal{A}_{(3)} \\
        \vdots & \vdots & \ddots & \vdots \\
        \mathcal{A}_{(p)} & \mathcal{A}_{(p-1)} & \dots & \mathcal{A}_{(1)}
    \end{array}\right]\in \mathbb{C}^{mp\times np}.\]
    The \textbf{unfolding} of $\mathcal{A}$ as the block matrix 
    \[\mathrm{unfold}(\mathcal{A}) := \left[\begin{array}{c}
        \mathcal{A}_{(1)} \\
        \mathcal{A}_{(2)} \\
        \vdots \\
        \mathcal{A}_{(p)}
    \end{array}\right]\in \mathbb{C}^{mp\times n}\]
\end{defn}
\begin{defn}[The $t$-product \cite{kilmer2000third}]
    For order 3 hypermatrices $\mathcal{A}\in \mathbb{C}^{m\times n\times p}$ and $\mathcal{B}\in \mathbb{C}^{n\times q\times p}$, the \textbf{t-product} of $\mathcal{A}$ and $\mathcal{B}$, denoted $\mathcal{A}*_t\mathcal{B}$, is the an order $3$ hypermatrix given by 
    \[A*_tB := \mathrm{fold}\big(\mathrm{bcirc}(A)\mathrm{unfold}(B)\big)\in \mathbb{C}^{m\times q\times p},\]
    where $\mathrm{fold}$ is defined to be the inverse of $\mathrm{unfold}$. 
\end{defn}
\begin{defn}
    For an order 3 hypermatrix $\mathcal{A}\in \mathbb{C}^{m\times n\times p}$, the \textbf{$\boldsymbol{(i,j)}$-tube} of $\mathcal{A}$ is the vector 
    \[\mathcal{A}[i,j,:] = [a_{ijk}]_{1\times 1\times p}\in \mathbb{C}^p\]
    (that is, the pair of indices $(i,j)$ is fixed). 
\end{defn}
The $t$-product of hypermatrices $\mathcal{A}\in \mathbb{C}^{m\times n\times p}$ and $\mathcal{B}\in \mathbb{C}^{n\times q\times p}$ may then be redefined tubewise as 
\begin{equation}\label{tubewise t-product}
    (\mathcal{A}*_t\mathcal{B})[i,j,:] = \sum\limits_{\ell=1}^n\mathcal{A}[i,\ell,:]\circledast\mathcal{B}[\ell,j,:],
\end{equation}
where $\circledast$ denotes the discrete circular convolution operator \cite{kilmer2013third}. The tubewise characterization of the $t$-product is useful because applying the Fast Fourier Transform on both sides of equation \eqref{tubewise t-product}, we obtain 
\begin{equation}\label{tubewise t-product in frequency domain}
    \mathrm{fft}\big((\mathcal{A}*_t\mathcal{B})[i,j,:]\big) = \sum\limits_{\ell=1}^n\mathrm{fft}\big(\mathcal{A}[i,\ell,:]\big)\odot\mathrm{fft}\big(\mathcal{B}[\ell,j,:]\big)
\end{equation}
with $\odot$ denoting the Hadamard product, from which it straightforwardly follows that slice-wise under this Fourier change-of-basis transformation the $t$-product reduces to regular matrix multiplication. Formally, we have the following fact. 
\begin{fact}\cite{kilmer2013third}\label{t-product in frequency domain slicewise reduces to matrix multiplication}
    Let $\mathcal{A}\in \mathbb{C}^{m\times n\times p}$, and $\mathcal{B}\in \mathbb{C}^{n\times q\times p}$, and denote $\widehat{\mathcal{A}} := \mathrm{fft}_3(\mathcal{A})$, $\widehat{\mathcal{B}} := \mathrm{fft}_3(\mathcal{B})$, and $\widehat{\mathcal{A}*_t\mathcal{B}} := \mathrm{fft}_3(\mathcal{A}*_t\mathcal{B})$. Then 
    \[\widehat{\mathcal{A}*_t\mathcal{B}}_{(\ell)} = \widehat{\mathcal{A}}_{(\ell)}\widehat{\mathcal{B}}_{(\ell)}\]
    for each $\ell\in [p]$. 
\end{fact}
Equivalently, we have the following fact, which will also be useful for our purposes:
\begin{fact}\cite{kilmer2013third}\label{bcirc property}
    Let $\mathcal{A}\in \mathbb{C}^{m\times n\times p}$ and $\widehat{\mathcal{A}} := \mathrm{fft}_3(\mathcal{A})$. Then 
    \[\mathrm{bcirc}(\mathcal{A}) = (\mathrm{DFT}_p^H\otimes I_m)\mathrm{diag}\big(\widehat{\mathcal{A}}_{(1)},...,\widehat{\mathcal{A}}_{(p)}\big)(\mathrm{DFT}_p\otimes I_n),\]
    where $\mathrm{DFT}_p$ is the $p\times p$ normalized discrete Fourier transform given by 
    \[\mathrm{DFT}_p[i,j] = \frac{1}{\sqrt{p}}\omega_p^{(i-1)(j-1)}\] 
    with $\omega_p := e^{-\frac{2\pi i}{p}}$ and $i,j\in [p]$, and $H$ here denoting the normal matrix conjugate transpose. 
\end{fact}

\begin{proposition}[t-Product Algebra \cite{kilmer2011factorization}]
    The vector space of hypermatrices $\mathbb{C}^{n\times n\times p}$ is an associative unital algebra under the $t$-product $*_t$, with unit given by $\mathcal{I}_1$ whose first frontal slice is the identity matrix $I_n$, and whose other frontal slices are the zero matrix $0_{n\times n}$; that is, 
    \[\mathrm{unfold}(\mathcal{I}_1) = \left[\begin{array}{c}
        I_n \\
        0_{n\times n} \\
        \vdots \\
        0_{n\times n}
    \end{array}\right].\]
\end{proposition}

\begin{defn}[t-Conjugate Transpose \cite{kilmer2011factorization}]\label{t-conjugate transpose}
    Let $\mathcal{A}\in \mathbb{C}^{m\times n\times p}$. The \textbf{$\boldsymbol{t}$-conjugate transpose} of $\mathcal{A}$ is the hypermatrix $\mathcal{A}^H\in \mathbb{C}^{n\times m\times p}$ obtained by taking the conjugate transpose of each of the frontal slices $\mathcal{A}_{(k)}$ and then reversing their order from $k=2$ to $k=p$; that is, 
    \[\mathcal{A}^H := \mathrm{fold}\left(\left[\begin{array}{c}
        \mathcal{A}_{(1)}^H \\
        \mathcal{A}_{(p)}^H \\
        \vdots \\
        \mathcal{A}_{(2)}^H
    \end{array}\right]\right)\]
    A hypermatrix $\mathcal{A} \in \mathbb{C}^{n\times n\times p}$ is \textbf{t-Hermitian} if $\mathcal{A}^H = \mathcal{A}$.
\end{defn}
The ordering of the frontal slices in the definition of the $t$-conjugate transpose is so that in general 
\[(\mathcal{A}*_t\mathcal{B})^H = \mathcal{B}^H*_t\mathcal{A}^H\]
for any $\mathcal{A}\in \mathbb{C}^{m\times n\times p}$ and $\mathcal{B}\in \mathbb{C}^{n\times q\times p}$. Thus, equipped with the $t$-conjugate transpose, $\mathbb{C}^{n\times n\times p}$ is an associative unital $^*$-algebra. For our purposes, it will be important to note how the $t$-conjugate transpose transforms a hypermatrix after the Fast Fourier Transform has been applied to its tubes. Indeed, we have the following fact:
\begin{fact}\label{t-conjugate transpose in frequency domain}
    Let $\mathcal{A}\in \mathbb{C}^{m\times n\times p}$ and $\widehat{\mathcal{A}} := \mathrm{fft}_3(\mathcal{A})$. Then 
    \[\mathcal{A}^H = \mathrm{ifft}_3\Big(\widehat{\mathcal{A}}^{\widetilde{H}}\Big)\]
    where $\widetilde{H}$ means we take the conjugate transpose of each frontal slice without reversing their order, i.e. 
    \[\widehat{\mathcal{A}}^{\widetilde{H}} := \mathrm{fold}\left(\left[\begin{array}{c}
        \widehat{\mathcal{A}}_{(1)}^H \\
        \widehat{\mathcal{A}}_{(2)}^H \\
        \vdots \\
        \widehat{\mathcal{A}}_{(p)}^H
    \end{array}\right]\right).\]
    Consequently, while in the \textbf{spatial domain} (that is, before applying $\mathrm{fft}_3$) we have 
    \[\left(\mathcal{A}^H\right)_{(\ell)} = (\mathcal{A}_{(p-\ell+2)})^H\]
    for each $\ell\in [p]$, whereas in the \textbf{frequency domain} (that is, after applying $\mathrm{fft}_3$) we have 
    \begin{equation}\label{conjugate-transpose in frequency domain}
        \left(\widehat{\mathcal{A}^H}\right)_{(\ell)} = (\widehat{\mathcal{A}}_{(\ell)})^H
    \end{equation}
    for each $\ell\in [p]$. 
\end{fact}
See Appendix Section B for proof of this fact. Lastly, we introduce the notion of "Fourier conjugation," which reduces to conjugation slice-wise in the frequency domain. 
\begin{defn}\label{Fourier conjugation}
    Let $\mathcal{A}\in \mathbb{C}^{m\times n\times p}$ denote $\overline{\mathcal{A}}\in \mathbb{C}^{m\times n\times p}$ be such that \[\overline{\mathcal{A}}[i,j,k] = \overline{\mathcal{A}[i,j,k]}\]
    for each $i\in [m]$, $j\in [n]$, and $k\in [p]$; that is, $\overline{\mathcal{A}}$ is obtained from $\mathcal{A}$ by conjugated every entry in $\mathcal{A}$. We define \textbf{Fourier conjugation} of $\mathcal{A}$ to be the hypermatrix $J(\mathcal{A})\in \mathbb{C}^{m\times n\times p}$ such that 
    \[J(\mathcal{A})_{(\ell+1)} := \overline{\mathcal{A}_{\big((-\ell\bmod p)+1\big)}}\]
    for $\ell=0,1,...,p-1$. 
\end{defn}
\begin{fact}\label{Fourier conjugation in frequency domain}
    Let $\mathcal{A}\in \mathbb{C}^{m\times n\times p}$ and $\widehat{\mathcal{A}} := \mathrm{fft}_3(\mathcal{A})$. Then 
    \[\widehat{J(\mathcal{A})}_{(\ell)} = \overline{\widehat{\mathcal{A}}_{(\ell)}}\]
    for each $\ell\in [p]$. 
\end{fact}
See Appendix Section B for proof of this fact. 

\subsection{Degree $k$ $t$-Hermitian Forms}
\begin{defn}
    Let $\mathcal{A}\in \mathbb{C}^{n\times...\times n\times p}$ be an order $2k+1$ hypermatrix. The \textbf{$\boldsymbol{\ell^{th}}$ frontal slice} of $\mathcal{A}$ is the hypermatrix 
    \[\mathcal{A}_{(\ell)} := \mathcal{A}[:,\dots,:,\ell]\in \mathbb{C}^{\overbrace{n\times\dots\times n}^{2k\text{ times}}}\]
    (that is, the last index $\ell$ is fixed). The \textbf{$\mathbf{(i_1,\dots,i_k,j_1,\dots,j_k)^{th}}$ tube} of $\mathcal{A}$ is the vector 
    \[a_{i_1,\dots,i_k,j_1,\dots,j_k} := \mathcal{A}[i_1,\dots,i_k,j_1,\dots,j_k,:]\in \mathbb{C}^p\]
    (that is, indices $(i_1,\dots,i_k,j_1,\dots,j_k)$ are fixed). In general, we will always refer to the last dimension of a hypermatrix as its \textbf{tubal mode}. 
\end{defn}
\begin{defn}
    The $n\times p$ matrix $\mathcal{X}$ is an \textbf{indeterminate complex tubal vector} if 
    \[\widehat{\mathcal{X}}_{(\ell)} = x^{\ell}\qquad \forall \ell\in [p],\]
    with $\widehat{\mathcal{X}}$ obtained after FFT applied to the tubal mode of $\mathcal{X}$ and $x^{\ell} = \left[\begin{array}{ccc}
        x_1^{\ell} & \dots & x_n^{\ell}
    \end{array}\right]^T$ an $n$-dimensional vector of complex variables. That is, in the frequency domain each of the columns of $\mathcal{X}$ are vectors of complex variables. 
\end{defn}
\begin{defn}[t-Multilinear Hypermatrix Product]
    Let $N\geq 2$ and for each $j\in [N]$, let $\mathcal{X}^j$ be an $n_j\times p$ indeterminate complex tubal vector. Suppose also that $\mathcal{A}\in \mathbb{C}^{n_1\times\dots \times n_N\times p}$ is an order $N+1$ hypermatrix. The \textbf{(right) t-multilinear hypermatrix product} of $\mathcal{A}$ with the tuple $(\mathcal{X}^1,...,\mathcal{X}^N)$ is given by the $p$-dimensional tube
    \[\mathcal{A}*_t(\mathcal{X}^1,...,\mathcal{X}^N) := \sum\limits_{i_1,\dots,i_N=1}^{n_1,\dots,n_N}a_{i_1\dots i_N}\circledast \mathcal{X}^1[i_1,:]\circledast\dots\circledast\mathcal{X}^N[i_N,:]\]
    with "$\circledast$" denote the discrete circular convolution operator. 
\end{defn}
\begin{remark}\label{t-multilinear hypermatrix product in frequency domain}
    Denote $\mathcal{A}' := \mathcal{A}*_t(\mathcal{X}^1,\dots,\mathcal{X}^N)$. Then viewing the tubal vectors $\mathcal{X}^j$ as $n_j\times 1\times p$ order $3$ hypermatrices, from equations \eqref{tubewise t-product} and \eqref{tubewise t-product in frequency domain}, it follows that 
    \[\mathrm{fft}(\mathcal{A}')[\ell] = \sum\limits_{i_1,\dots,k_i=1}^{n_1,\dots,n_N}\widehat{\mathcal{A}}[i_1,\dots,i_N,\ell]\widehat{\mathcal{X}^1}[i_1,1,\ell]\cdot...\cdot\widehat{\mathcal{X}^N}[i_N,1,\ell]\qquad \forall \ell\in [p],\]
    with $\widehat{\mathcal{A}}:=\mathrm{fft}_{2k+1}(\mathcal{A})$ and $\widehat{\mathcal{X}^j}:=\mathrm{fft}_3(\mathcal{X}^j)$ for each $j\in [N]$. Consequently, by definition 
    \begin{equation}\label{slicewise t-MMM}
        \mathrm{fft}(\mathcal{A}')[\ell] = \widehat{\mathcal{A}}_{(\ell)}*\left(\widehat{\mathcal{X}^1}_{(\ell)},...,\widehat{\mathcal{X}^N}_{(\ell)}\right)
    \end{equation}
    with "$*$" denoting regular multilinear matrix multiplication, for each $\ell\in [p]$. 
\end{remark}
\begin{defn}[t-Conjugate Partially Symmetric]
    Let $\mathcal{A}\in \mathbb{C}^{n\times...\times n\times p}$ be an order $2k+1$ hypermatrix and define $\widehat{\mathcal{A}} := \mathrm{fft}_{2k+1}(\mathcal{A})$. We say that $\mathcal{A}$ is \textbf{t-conjugate partially symmetric} if each frontal slice in the frequency domain, $\widehat{\mathcal{A}}_{(\ell)}$, is a conjugate partially symmetric hypermatrix.  
\end{defn}
\begin{defn}
    Let $\mathcal{A}\in \mathbb{C}^{n\times\dots\times n\times p}$ be an order $2k+1$ $t$-conjugate partially symmetric hypermatrix and $\mathcal{X}$ be an $n\times p$ indeterminate complex tubal vector, which we may regard as an $n\times 1\times p$ order $3$ hypermatrix. The \textbf{degree $\boldsymbol{k}$ t-Hermitian form associated with $\boldsymbol{\mathcal{A}}$} is defined as the $t$-multilinear hypermatrix product
    \[h_{\mathcal{A}}(\mathcal{X}) := \mathcal{A}*_t(\underbrace{J(\mathcal{X}),...,J(\mathcal{X})}_{k\text{ times}},\overbrace{\mathcal{X},...,\mathcal{X}}^{k\text{ times}}).\] 
\end{defn}
Since $\mathcal{A}$ is assumed to be $t$-conjugate partially symmetric, each frontal slice of $\widehat{\mathcal{A}}$ is conjugate partially symmetric, hence by equation \eqref{slicewise t-MMM}, in the frequency domain each coordinate of $h_{\mathcal{A}}(\mathcal{X})$ defines a unique degree $k$ Hermitian form. Formally, we have the following fact:
\begin{fact}\label{Degree k t-Hermitian form as a collection of degree k Hermitian forms}
    After applying the FFT along tubal modes of $\mathcal{A}$, $\mathcal{X}$, and $J(\mathcal{X})$, a degree $k$ $t$-Hermitian form decomposes into a collection of classical degree $k$ Hermitian forms; that is, 
    \[\widehat{h_{\mathcal{A}}(\mathcal{X})}[\ell] = \widehat{\mathcal{A}}_{(\ell)}*\left(\underbrace{\overline{x^{\ell}},\dots,\overline{x^{\ell}}}_{k\text{ times}},\overbrace{x^{\ell},\dots,x^{\ell}}^{k\text{ times}}\right)\qquad \forall \ell\in [p],\]
    with $\widehat{h_{\mathcal{A}}(\mathcal{X})} := \mathrm{fft}\big(h_{\mathcal{A}}(\mathcal{X})\big)$. 
\end{fact}
Since degree $k$ Hermitian forms uniquely corresponds to an order $2k$ conjugate partially symmetric hypermatrices (Fact \ref{Hermitian Form Bijection}), by definition each degree $k$ $t$-Hermitian form uniquely corresponds to an order $2k+1$ $t$-conjugate partially symmetric hypermatrix. Formally, we summarize with the following fact. 
\begin{fact}\label{t-Hermitian Form Bijection}
    There is a bijection between order $2k+1$ $t$-conjugate partially symmetric hypermatrices $\mathcal{A}\in \mathbb{C}^{n\times\dots\times n\times p}$ and degree $k$ $t$-Hermitian forms in the $n\times p$ indeterminate complex tubal vector $\mathcal{X}$, and in particular it is given by the map 
    \[\mathcal{A}\mapsto h_{\mathcal{A}}(\mathcal{X}) := \mathcal{A}*_t\big(J(\mathcal{X}),\dots,J(\mathcal{X}),\mathcal{X},\dots,\mathcal{X}\big)\]
\end{fact}
\begin{example}[Degree $1$ $t$-Hermitian Forms]\label{Degree 1 t-Hermitian Forms}
    In \cite{qi2021t}, L. Qi and X. Zhang define the \textbf{t-quadratic form} of a t-symmetric hypermatrix $\mathcal{A}\in \mathbb{R}^{n\times n\times p}$ (t-symmetric meaning $\mathcal{A} = \mathcal{A}^T$, where $T$ similarly defined as $H$ in Definition \ref{t-conjugate transpose}) as the $t$-product of order $3$ hypermatrices
    \[\mathcal{X}^T*_t\mathcal{A}*_t\mathcal{X},\]
    with $\mathcal{X}$ is an indeterminate complex tubal vector viewed as an $n\times 1\times p$ order $3$ hypermatrix. In the degree $k=1$ case, each $t$-Hermitian hypermatrix $\mathcal{A}\in \mathbb{C}^{n\times n\times p}$ uniquely defines the $t$-Hermitian form 
    \[h_{\mathcal{A}}(\mathcal{X}) = \mathcal{A}*_t(J(\mathcal{X}),\mathcal{X})\]
    with $\mathcal{X}$ an indeterminate complex tubal vector of indeterminates, and we claim that similarly 
    \[\mathcal{A}*_t(J(\mathcal{X}),\mathcal{X}) = \mathcal{X}^H*_t\mathcal{A}*_t\mathcal{X}\]
    with $H$ denoting the $t$-conjugate transpose from Definition \ref{t-conjugate transpose} and $*_t$ on the right-hand side denoting the usual $t$-product on order $3$ hypermatrices. That is, we obtain a complex $t$-analogue of Equation \eqref{Relationship between MM and MMM}. 
    
    To this end, first note that if $A$, $B$, and $C$ are $n\times n$ matrices, then 
    \begin{equation}\label{triple matrix product}
        (ABC)_{ij} = \sum\limits_{k_1,k_2=1}^nA_{ik_1}B_{k_1k_2}C_{k_2j}
    \end{equation}
    for each $i,j$. Therefore setting $\mathcal{A}' := \mathcal{A}*_t\left(J(\mathcal{X}),\mathcal{X})\right)$, by the definition of t-multilinear hypermatrix multiplication/Fact \ref{t-multilinear hypermatrix product in frequency domain}, for each $\ell\in [p]$ we have that 
    \begin{align*}
        \widehat{\mathcal{A}'}[\ell] &= \sum\limits_{k_1,k_2=1}^n\widehat{\mathcal{A}}[k_1,k_2,\ell]\widehat{J(\mathcal{X})}[k_1,1,\ell]\widehat{\mathcal{X}}[k_2,1,\ell] \\
        &= \sum\limits_{k_1,k_2=1}^n\overline{\widehat{\mathcal{X}}[k_1,1,\ell]}\widehat{\mathcal{A}}[k_1,k_2,\ell]\widehat{\mathcal{X}}[k_2,1,\ell],\quad \text{by Fact }\eqref{Fourier conjugation in frequency domain} \\
        &= \sum\limits_{k_1,k_2=1}^n\left(\widehat{\mathcal{X}}\right)^H[1,k_1,\ell]\widehat{\mathcal{A}}[k_1,k_2,\ell]\widehat{\mathcal{X}}[k_2,1,\ell] \\
        &= \sum\limits_{k_1,k_2=1}^n\widehat{\mathcal{X}^H}[1,k_1,\ell]\widehat{\mathcal{A}}[k_1,k_2,\ell]\widehat{\mathcal{X}}[k_2,1,\ell],\quad \text{by equation }\eqref{conjugate-transpose in frequency domain}.
    \end{align*}
    Therefore by equation \eqref{triple matrix product}, for each $l\in [p]$ we have that
    \[\widehat{\mathcal{A}'}[\ell] = \widehat{\mathcal{X}^H}_{(\ell)}\widehat{\mathcal{A}}_{(\ell)}\widehat{\mathcal{X}}_{(\ell)} = \widehat{\left(\mathcal{X}^H*_t\mathcal{A}*_t\mathcal{X}\right)}[\ell],\]
    where the second equation is obtain by applying Fact \ref{t-product in frequency domain slicewise reduces to matrix multiplication} twice. 
    Thus, 
    \[\widehat{\mathcal{A}'} = \widehat{\left(\mathcal{X}^H*_t\mathcal{A}*_t\mathcal{X}\right)},\]
    and so by applying the inverse FFT on both sides we obtain 
    \[\mathcal{A}' = \mathcal{X}^H*_t\mathcal{A}*_t\mathcal{X}.\]
\end{example}

\subsubsection{Intrinsic Definition of $t$-Hermitian Forms}
The $t$-product allows tubal vectors to be viewed as vectors over an algebra of tubal scalars; see \cite{braman2010third,kilmer2013third} for an in depth treatment of this module-theoretic interpretation of the $t$-product. More generally, sesquilinear forms and Hermitian forms may be defined as ring/algebra-valued maps on modules over said rings/algebras equipped with involution; see for instance \cite{bayer2014sesquilinear}. Here, we combine this standard algebra-valued viewpoint with the higher bi-degree construction of sesquilinear forms provided in Subsection \ref{High-Level Definition of Hermitian Forms} to give an intrinsic definition of $t$-Hermitian forms. 

Let $\mathbb{T}_p := \mathbb{C}^{1\times 1\times p}\cong \mathbb{C}^p$ denote the space of $p$-dimensional tubes. For tubes, the $t$-product reduces to the discrete circular convolution operator: 
\[a*_tb = a\circledast b\qquad \forall a,b\in \mathbb{T}_p.\]
Equipped with this product, the unit $\delta := (1,0,\dots,0)$, and the involution $a^* := J(a)$ (from Definition \ref{Fourier conjugation}, viewing $a$ as a $1\times 1\times p$ order $3$ hypermatrix), the space $\mathbb{T}_p$ is a commutative unital $^*$-algebra. Therefore we may consider the free rank-$n$ module $\mathbb{T}_p^n \cong \mathbb{C}^{n\times 1\times p}$ over $\mathbb{T}_p$. 
\begin{defn}
    A \textbf{$\boldsymbol{\mathbb{T}_p}$-valued $\boldsymbol{(k,k)}$-sesquilinear form} is a map 
    \[S:\overbrace{\mathbb{T}_p^n\times\dots\times \mathbb{T}_p^n}^{2k\text{ times}}\rightarrow \mathbb{T}_p\]
    that is $\mathbb{T}_p$-conjugate linear in its first $k$ arguments and $\mathbb{T}_p$-linear in its last $k$ arguments; that is, 
    for any $a,b\in \mathbb{T}_p$ and $\mathcal{U}_i,\mathcal{V}_i,\mathcal{W}_i\in \mathbb{T}_p^n$, 
    \begin{align*}
        &S_{\mathcal{A}}(\mathcal{U}_1,\dots,\mathcal{U}_i*_ta+\mathcal{W}_i*_tb,\dots,\mathcal{U}_k,\mathcal{V}_1,\dots,\mathcal{V}_k) \\
        &\quad = J(a)*_tS_{\mathcal{A}}(\mathcal{U}_1,\dots,\mathcal{U}_i,\dots\mathcal{U}_k,\mathcal{V}_1,\dots,\mathcal{V}_k) \\
        &\quad + J(b)*_tS_{\mathcal{A}}(\mathcal{U}_1,\dots,\mathcal{W}_i,\dots\mathcal{U}_k,\mathcal{V}_1,\dots,\mathcal{V}_k)
    \end{align*}
    and 
    \begin{align*}
        &S_{\mathcal{A}}(\mathcal{U}_1,\dots,\mathcal{U}_k,\mathcal{V}_1,\dots\mathcal{V}_i*_ta+\mathcal{W}_i*_tb,\dots,\mathcal{V}_k) \\
        &\quad = a*_tS_{\mathcal{A}}(\mathcal{U}_1,\dots\mathcal{U}_k,\mathcal{V}_1,\dots,\mathcal{V}_i,\dots\mathcal{V}_k) \\
        &\quad + b*_tS_{\mathcal{A}}(\mathcal{U}_1,\dots\mathcal{U}_k,\mathcal{V}_1,\dots,\mathcal{W}_i,\dots,\mathcal{V}_k).
    \end{align*}
    We say that $S$ is \textbf{separately symmetric} if 
    \[S(\mathcal{U}_{\sigma(1)},\dots,\mathcal{U}_{\sigma(k)};\mathcal{V}_{\tau(1)},\dots,\mathcal{V}_{\tau(k)}) = S(\mathcal{U}_1,\dots,\mathcal{U}_k;\mathcal{V}_1,\dots,\mathcal{V}_k)\]
    for all $\mathcal{U}_1,\dots,\mathcal{U}_k,\mathcal{V}_1,\dots,\mathcal{V}_k\in \mathbb{T}_p^n$ and every $\sigma,\tau\in S_k$. We say that $S$ is \textbf{$\boldsymbol{t}$-Hermitian} if 
    \[J\big(S(\mathcal{V}_1,\dots,\mathcal{V}_k;\mathcal{U}_1,\dots,\mathcal{U}_k)\big) = S(\mathcal{U}_1,\dots,\mathcal{U}_k;\mathcal{V}_1,\dots,\mathcal{V}_k)\]
    for all $\mathcal{U}_1,\dots,\mathcal{U}_k,\mathcal{V}_1,\dots,\mathcal{V}_k\in \mathbb{T}_p^n$. 
\end{defn}
\begin{remark}
    As in the classical setting, separate symmetry allows $S$ to be viewed as an ordinary two-argument sesquilinear form 
    \[\widetilde{S}:\mathrm{Sym}^k(\mathbb{T}_p^n)\times \mathrm{Sym}^k(\mathbb{T}_p^n)\rightarrow \mathbb{T}_p\]
    satisfying 
    \[\widetilde{S}(\mathcal{U}_1\odot\dots\odot \mathcal{U}_k; \mathcal{V}_1\odot\dots\odot \mathcal{V}_k) = S(\mathcal{U}_1,\dots,\mathcal{U}_k; \mathcal{V}_1,\dots,\mathcal{V}_k)\]
    with $\odot$ denoting the symmetric tensor product. Moreover, $S$ is Hermitian if and only if $\widetilde{S}$ is Hermitian in the usual two-argument sense with respect to the involution $J$. 
\end{remark}
\begin{defn}
    Let $S$ be a separately symmetric $t$-Hermitian $\mathbb{T}_p$-valued $(k,k)$-sesquilinear form on $\mathbb{T}_p^n$. The map 
    \begin{align*}
        h_S:\mathbb{T}_p^n&\longrightarrow \mathbb{T}_p \\
        \mathcal{V}&\mapsto S(\mathcal{V},\dots,\mathcal{V};\mathcal{V},\dots,\mathcal{V})
    \end{align*}
    is called the \textbf{degree $\boldsymbol{k}$ $\boldsymbol{t}$-Hermitian form} associated with $S$. 
\end{defn}
\begin{remark}
    Equivalently, in terms of the induced Hermitian sesquilinear form $\widetilde{S}$ on $\mathrm{Sym}^k(\mathbb{T}_p^n)$, 
    \[h_S(\mathcal{V}) = \widetilde{S}(\mathcal{V}^{\odot k},\mathcal{V}^{\odot k}).\]
    Note also that the $t$-Hermitian property of $S$ implies that 
    \[J\big(h_S(\mathcal{V})\big) = h_S(\mathcal{V})\]
    which by Fact \ref{Fourier conjugation in frequency domain} is implies that  
    \[\widehat{h_S(\mathcal{V})}[\ell]\in \mathbb{R}\]
    for each $\ell\in [p]$. 
\end{remark}
\begin{proposition}
    Let $\{\mathcal{E}_1,\dots,\mathcal{E}_n\}$ denote the standard ordered basis for $\mathbb{T}_p^n$ (i.e. $\mathcal{E}_i$ is the tubal vector with $\delta$ in $i^{th}$-coordinate and the $0$ tube for all of its other coordinates). Also, for a separately symmetric $t$-Hermitian $\mathbb{T}_p$-valued $(k,k)$-sesquilinear form $S$, define the tube 
    \[a_{i_1\dots i_k j_1\dots j_k} := S(\mathcal{E}_{i_1},\dots,\mathcal{E}_{i_k};\mathcal{E}_{j_1},\dots,\mathcal{E}_{j_k})\]
    and let $\mathcal{A}$ denote the order $2k+1$ $n\times\dots\times n\times p$ hypermatrix such that 
    \[\mathcal{A}[i_1,\dots,i_k,j_1,\dots,j_k,:] = a_{i_1\dots i_k j_1\dots j_k}\]
    for each $i_1,\dots,i_k,j_1,\dots,j_k\in [n]$. Then 
    \[S(\mathcal{U}_1,\dots,\mathcal{U}_k;\mathcal{V}_1,\dots,\mathcal{V}_k) = \mathcal{A}*_t\big(J(\mathcal{U}_1),\dots,J(\mathcal{U}_k);\mathcal{V}_1,\dots,\mathcal{V}_k\big),\]
    and so consequently $\mathcal{A}$ is $t$-conjugate partially symmetric. Conversely, every order $2k+1$ $t$-conjugate partially symmetric hypermatrix $\mathcal{A}\in \mathbb{C}^{n\times\dots\times n\times p}$ defines a unique separately symmetric $t$-Hermitian $\mathbb{T}_p$-valued $(k,k)$-sesquilinear form. 
\end{proposition}
\begin{proof}
    For each $r,s\in [k]$, we may write 
    \[\mathcal{U}_r = \sum\limits_{i=1}^n\mathcal{U}_r[i,:]*_t\mathcal{E}_i\quad\text{and}\quad \mathcal{V}_s = \sum\limits_{i=1}^n\mathcal{V}_s[i,:]*_t\mathcal{E}_i.\]
    Since among tubes the $t$-product reduces to the discrete circular convolution operator $\circledast$, by $\mathbb{T}_p$-conjugate linearity in the first $k$ arguments and $\mathbb{T}_p$-linearity in the last $k$ arguments, it follows that  $S(\mathcal{U}_1,\dots,\mathcal{U}_k;\mathcal{V}_1,\dots,\mathcal{V}_k)$ is equal to the sum 
    \[\sum\limits_{\substack{i_1,\dots,i_k = 1 \\ j_1,\dots,j_k=1}}^na_{i_1\dots i_k j_1\dots j_k}\circledast J(\mathcal{U}_1[i_1,:])\circledast\dots \circledast J(U_k[i_k,:])\circledast \mathcal{V}_1[j_1,:]\circledast\dots \circledast \mathcal{V}_k[j_k,:]\]
    which is precisely the $t$-multilinear hypermatrix product of $\mathcal{A}$ with $(J(\mathcal{U}_1),\dots,J(\mathcal{U}_k),\mathcal{V}_1,\dots,\mathcal{V}_k)$ (viewing $\mathcal{U}_1,\dots,\mathcal{U}_k,\mathcal{V}_1,\dots,\mathcal{V}_k$ as $n\times 1\times p$ order $3$ hypermatrices). Applying separate symmetry and $t$-Hermitian-ness to the basis vectors $\mathcal{E}_1,\dots,\mathcal{E}_n$, it follows that 
    \[J(a_{j_{\sigma(1)}\dots j_{\sigma(k)}}i_{\tau(1)}\dots i_{\tau(k)}) = a_{i_1\dots i_k j_1\dots j_k}\]
    for all $i_1,\dots,i_k,j_1,\dots,j_k\in [n]$ and every $\sigma,\tau\in S_k$, which is equivalent to 
    \[\overline{\widehat{a}_{j_{\sigma(1)}\dots j_{\sigma(k)}i_{\tau(1)}\dots i_{\tau(k)}}[\ell]} = \widehat{a}_{i_1\dots i_k j_1\dots j_k}[\ell]\qquad \forall \ell\in [p]\]
    for all $i_1,\dots,i_k,j_1,\dots,j_k\in [n]$ and every $\sigma,\tau\in S_k$, thus $\mathcal{A}$ is $t$-conjugate partially symmetric. 

    The converse straightforwardly follows by reversing the argument. In particular, conjugate partial symmetry of each of the frontal slices in the frequency domain implies, after applying the inverse FFT, both separate symmetry and the $t$-Hermitian identity of the induced form. 
\end{proof}
That is, every $\mathbb{T}_p$-valued separately symmetric $t$-Hermitian $(k,k)$-sesquilinear form determines a unique order $2k+1$ $t$-conjugate partially symmetric hypermatrix $\mathcal{A}\in \mathbb{C}^{n\times\dots\times n\times p}$, and vice versa. Since degree $k$ $t$-Hermitian forms are the diagonal of $\mathbb{T}_p$-valued separately symmetric $t$-Hermitian $(k,k)$-sesquilinear forms, every order $2k+1$ $t$-conjugate partially symmetric hypermatrix $\mathcal{A}\in \mathbb{C}^{n\times\dots\times n\times p}$ then defines a unique degree $k$ $t$-Hermitian form. 

The Fourier transform gives a particularly clean interpretation. By construction, for an order $2k+1$ $t$-conjugate partially symmetric hypermatrix $\mathcal{A}$, each frontal slice of $\mathcal{A}$ in the frequency domain, $\widehat{\mathcal{A}}_{(\ell)}$, defines an ordinary separately symmetric Hermitian $(k,k)$-sesquilinear form 
\[S_{\ell}:\overbrace{\mathbb{C}^n\times\dots\times\mathbb{C}^n}^{2k\text{ times}}\rightarrow \mathbb{C}\]
which satisfies
\begin{align*}
    \mathrm{fft}\big(S_{\mathcal{A}}(\mathcal{U}_1,\dots,\mathcal{U}_k;\mathcal{V}_1,\dots,\mathcal{V}_k)\big)[\ell] &= S_{\ell}\big((\widehat{U_1})_{(\ell)},\dots,(\widehat{U_k})_{(\ell)};(\widehat{\mathcal{V}_1})_{(\ell)},\dots,(\widehat{\mathcal{V}_k})_{(\ell)}\big) \\
    &= S_{\ell}(u_1^{\ell},\dots,u_k^{\ell};v_1^{\ell},\dots,v_k^{\ell})
\end{align*}
for some $u_1^{\ell},\dots,u_k^{\ell},v_1^{\ell},\dots,v_k^{\ell}\in \mathbb{C}^n$. In particular, on the diagonal, 
\[\mathrm{fft}\big(h_{\mathcal{A}}(\mathcal{V})\big)[\ell] = h_{\widehat{\mathcal{A}}_{(\ell)}}(\widehat{\mathcal{V}}_{(\ell)}) = h_{\widehat{\mathcal{A}}_{(\ell)}}(v^{\ell})\]
for some $v^{\ell}\in \mathbb{C}^n$, for each $\ell\in [p]$. Thus, under the $^*$-isomorphism
\begin{align*}
    \mathrm{fft}:(\mathbb{T}_p,*_t,J)&\longrightarrow (\mathbb{C}^p,\odot,\overline{(\cdot)}) \\
    a&\mapsto \mathrm{fft}(a) =: \widehat{a},
\end{align*}
a degree $k$ $t$-Hermitian form is equivalent to a collection of $p$ ordinary degree $k$ Hermitian forms; in particular, $t$-Hermitian forms are Fourier-packaged families of classical degree $k$ Hermitian forms. 

Conversely, given any collection of Hermitian degree $k$ forms $h_1,\dots,h_p$ in the complex variables $x_1,\dots,x_n$, by Fact \ref{Hermitian Form Bijection} each are uniquely represented by order $2k$ Hermitian conjugate symmetric hypermatrices $\mathcal{B}_1,\dots,\mathcal{B}_p$; setting $\mathcal{A}\in \mathbb{C}^{n\times\dots\times n\times p}$ to be the unique order $2k+1$ hypermatrix such that 
\[\mathcal{A}_{(\ell)} = \mathrm{ifft}(\mathcal{B}_{\ell}) \iff \widehat{\mathcal{A}}_{(\ell)} = \mathcal{B}_{\ell} \qquad \forall \ell\in [p],\]
we obtain a unique order $2k+1$ $t$-conjugate partially symmetric hypermatrix. 

Thus, in summary, $\mathbb{T}_p$-valued separately symmetric $t$-Hermitian $(k,k)$-sesquilinear forms, degree $k$ $t$-Hermitian forms, and order $2k+1$ $t$-Hermitian partially symmetric hypermatrices each uniquely determine one another. 

\subsection{Positive Definite t-Hermitian Forms}\label{Positive Definite t-Hermitian Forms}
Recall from above that the FFT yields an unital $^*$-algebra isomorphism 
\begin{align*}
    \mathrm{fft}:(\mathbb{T}_p,*_t,J)&\longrightarrow \big(\mathbb{C}^p,\odot,\overline{(\cdot)}\big) \\
    a&\mapsto \mathrm{fft}(a) := \widehat{a};
\end{align*}
in particular, 
\[\mathrm{fft}(\delta) = \mathbf{1},\quad \mathrm{fft}(a*_tb) = \widehat{a}\odot \widehat{b},\quad\text{and}\quad \mathrm{fft}(J(a)) = \overline{\widehat{a}}\]
(with $\delta := (1,0,\dots,0)$ and $\mathbf{1} := (1,1,\dots,1)$). 
\begin{defn}
    A tube $a\in \mathbb{T}_p$ is \textbf{Fourier positive} if $\widehat{a}[\ell]\geq 0$ for each $\ell\in [p]$. The \textbf{positive cone} of the tubal algebra $\mathbb{T}_p$ is the set 
    \[\mathbb{T}_p^+ := \{a\in \mathbb{T}_p:\widehat{a}[\ell]\geq 0\text{ for every }\ell\in [p]\}.\]
\end{defn}
\begin{remark}
    Note that $\mathbb{T}_p^+$ is a closed convex cone in the underlying real vector space of $\mathbb{T}_p$.
\end{remark} 
Let $\mathcal{A}\in \mathbb{C}^{n\times\dots\times n\times p}$ be a $t$-conjugate partially symmetric hypermatrix. From the intrinsic definition of $h_{\mathcal{A}}:\mathbb{T}_p^n\rightarrow \mathbb{T}_p$, we may view $t$-Hermitian forms as $\mathbb{T}_p$-valued Hermitian forms. A natural definition for positive $t$-Hermitian forms is therefore the following.
\begin{defn}
    Let $\mathcal{A}\in \mathbb{C}^{n\times\dots\times n\times p}$ be an order $2k+1$ $t$-conjugate partially symmetric hypermatrix. We say that corresponding $t$-Hermitian form $h_{\mathcal{A}}$ is \textbf{$\boldsymbol{t}$-Hermitian positive definite (semidefinite)} if 
    \[h_{\mathcal{A}}(\mathcal{Z})\in \mathbb{T}_p^+\setminus \{0\}\text{ }(\in \mathbb{T}_p^+)\]
    for all nonzero $\mathcal{Z}\in \mathbb{T}_p^n\cong \mathbb{C}^{n\times 1\times p}$. 
\end{defn}
\begin{defn}
    Let $x = \left[\begin{array}{ccc}
        x_1 & \dots & x_n
    \end{array}\right]^T$ be a vector of complex variables. We denote the indeterminate complex tubal vector $\mathcal{X}^{\Delta}$ such that 
    \[(\widehat{\mathcal{X}^{\Delta}})_{(\ell)} = x\qquad \forall \ell\in [p],\]
    and for any $z\in \mathbb{C}^n$, so that under evaluation only one vector input is required, and in particular
    \[\widehat{\mathrm{Ev}_z(\mathcal{X}^{\Delta})}_{(\ell)} = z\qquad \forall \ell\in [p].\]
    We will refer to such tubal vectors as \textbf{diagonal indeterminate complex tubal vectors}. 
\end{defn}
\begin{lemma}[Characterizing $t$-Hermitian Positivity]\label{Characterizing t-Hermitian Positivity}
    Let $\mathcal{A}\in \mathbb{C}^{n\times\dots\times n\times p}$ be an order $2k+1$ $t$-conjugate partially symmetric hypermatrix, and for any $z\in \mathbb{C}^n$ denote 
    \[\widehat{\mathcal{A}}_{(\ell)}*(\overline{z},\dots,\overline{z},z,\dots,z) := h_{\widehat{\mathcal{A}}_{(\ell)}}(z).\]
    Then 
    \[\widehat{h_{\mathcal{A}}\big(\mathrm{Ev}_z(\mathcal{X}^{\Delta})\big)}[\ell] = h_{\widehat{\mathcal{A}}_{(\ell)}}(z)\qquad \forall \ell\in [p].\]
    That is, $\widehat{h_{\mathcal{A}}\big(\mathrm{Ev}_z(\mathcal{X}^{\Delta})\big)}[\ell]$ is the evaluated value of the classical Hermitian form corresponding to the $\ell^{th}$ frontal slice of $\mathcal{A}$ in the frequency domain. Moreover, $h_{\mathcal{A}}$ is $t$-Hermitian positive definite (semidefinite) if and only if  
    \[\widehat{h_{\mathcal{A}}\big(\mathrm{Ev}_z(\mathcal{X}^{\Delta})\big)}\in \mathbb{R}_{>0}^p\text{ }(\in \mathbb{R}_{\geq 0}^p),\]
    or equivalently, if and only if 
    \[\widehat{\mathcal{A}}_{(\ell)}*(\overline{z},\dots,\overline{z},z,\dots,z) > 0\text{ }(\geq 0)\qquad \forall \ell\in [p]\]
    for every nonzero $z\in \mathbb{C}^n$. 
\end{lemma}
\begin{proof}
    From Fact \ref{Degree k t-Hermitian form as a collection of degree k Hermitian forms}, 
    \[\widehat{h_{\mathcal{A}}(\mathcal{X}^{\Delta})}[\ell] = \widehat{\mathcal{A}}_{(\ell)}*(\overline{x},\dots,\overline{x},x,\dots,x)\qquad \ell\in [p]\]
    for some vector of complex variables $x$; hence, under evaluation, it follows that $\widehat{h_{\mathcal{A}}\big(\mathrm{Ev}_z(\mathcal{X}^{\Delta})\big)}[\ell] = h_{\widehat{\mathcal{A}}_{(\ell)}}(z)$ for each $\ell\in [p]$.

    Suppose $h_{\mathcal{A}}$ is $t$-Hermitian positive definite; that is, for any nonzero $\mathcal{Z}\in \mathbb{T}_p^n\cong \mathbb{C}^{n\times 1\times p}$, $h_{\mathcal{A}}(\mathcal{Z})\in \mathbb{T}_p^+\setminus \{0\}$. Fix an arbitrary $\ell_*\in [p]$ and nonzero vector in $z\in \mathbb{C}^n$. Since $\mathrm{fft}$ lifts to an isomorphism $\mathbb{T}_p^n\rightarrow \mathbb{C}^{n\times 1\times p}$ is an isomorphism, it follows that there exists a unique complex tubal vector $\mathcal{Z}(\ell_*,z)\in \mathbb{T}_p^n$ such that 
    \[\widehat{\mathcal{Z}(\ell_*,z)}_{(\ell)} = 
    \begin{cases}
        \begin{rcases}
            z, & \ell = \ell_* \\
            0, & \ell\neq \ell_*
        \end{rcases}.
    \end{cases}\]
    Consequently, 
    \[\widehat{h_{\mathcal{A}}\big(\mathcal{Z}(\ell_*,z)\big)}[\ell] = 
    \begin{cases}
        \begin{rcases}
            h_{\widehat{\mathcal{A}}_{(\ell_*)}}(z), & \ell = \ell_* \\
            0, & \ell\neq \ell_*
        \end{rcases}.
    \end{cases}\]
    Thus, in the frequency domain, $h_{\mathcal{A}}\big(\mathcal{Z}(\ell_*,z)\big)$ has only one coordinate that can possibly be nonzero, namely the $\ell_*^{th}$. Furthermore, since by assumption $h_{\mathcal{A}}\big(\mathcal{Z}(\ell_*,z)\big)\in \mathbb{T}_p^+\setminus \{0\}$, it follows that the $\ell_*^{th}$-coordinate is strictly positive. Since $z$ was an arbitrary nonzero vector, it then follows that 
    \[h_{\widehat{\mathcal{A}}_{(\ell_*)}}(z) > 0\]
    for every nonzero $z\in \mathbb{C}^n$; additionally, since $\ell_*\in [p]$ was arbitrary, in fact it follows that 
    \[h_{\widehat{\mathcal{A}}_{(\ell)}}(z) > 0\]
    for every nonzero $z\in \mathbb{C}^n$, for each $\ell\in [p]$. 

    Conversely, suppose 
    \[\widehat{h_{\mathcal{A}}\big(\mathrm{Ev}_z(\mathcal{X}^{\Delta})\big)}\in \mathbb{R}_{>0}^p\]
    for every nonzero $z\in \mathbb{C}^n$, which again from Fact \ref{Degree k t-Hermitian form as a collection of degree k Hermitian forms} is equivalent to 
    \[\widehat{\mathcal{A}}_{(\ell)}*(\overline{z},\dots,\overline{z},z,\dots,z) > 0\qquad \forall \ell\in [p].\]
    Let $\mathcal{Z}\in \mathbb{T}_p^n$ be an arbitrary nonzero complex tubal vector, and denote $z^{\ell} := \widehat{\mathcal{Z}}_{(\ell)}$ for each $\ell\in [p]$. If $z^{\ell}=0$, then clearly $h_{\widehat{\mathcal{A}}_{\ell}}(z^{\ell}) = 0$, however if $z^{\ell}\neq 0$ then by assumption 
    \[h_{\widehat{\mathcal{A}}_{(\ell)}}(z^{\ell}) = \widehat{\mathcal{A}}_{(\ell)}*(\overline{z^{\ell}},\dots,\overline{z^{\ell}},z^{\ell},\dots,z^{\ell}) > 0.\]
    Therefore, 
    \[\widehat{h_{\mathcal{A}}(\mathcal{Z})}[\ell] \geq 0\]
    for every $\ell\in [p]$. Moreover, since $\mathcal{Z}$ is nonzero and $\mathrm{fft}$ lifts to an isomorphism $\mathbb{T}_p^n\rightarrow \mathbb{C}^{n\times 1\times p}$, there exists atleast one $\ell_*\in [p]$ such that $z^{\ell_*}$ is nonzero, hence 
    \[\widehat{h_{\mathcal{A}}(\mathcal{Z})}[\ell_*] = h_{\widehat{\mathcal{A}}_{(\ell_*)}}(z^{\ell_*}) > 0.\]
    Thus, $h_{\mathcal{A}}(\mathcal{Z})\in \mathbb{T}_p^+\setminus \{0\}$. Since $\mathcal{Z}$ was arbitrary, it follows that $h_{\mathcal{A}}$ is $t$-Hermitian positive definite. 

    The $t$-Hermitian positive semidefinite case follows by the same argument, replacing strict inequalities with weak inequalities and omitting the requirement that the output tube be nonzero. 
\end{proof}
With this, we may now extend the Spectral Theorem of classical Hermitian forms to $t$-Hermitian forms. 
\begin{theorem}[Spectral Characterization of t-Hermitian Positive Definiteness]\label{Higher order t-Hermitian positive definite characterization}
    Let $\mathcal{A}\in \mathbb{C}^{n\times...\times n\times p}$ be an order $2k+1$ $t$-conjugate partially symmetric hypermatrix. Then $\mathcal{A}$ is t-Hermitian positive definite (semi-definite) if and only if $\widehat{\mathcal{A}}_{(\ell)}$ is Hermitian positive definite (semi-definite) for each $\ell\in [p]$, which is true if and only if every $\widehat{H}$ eigenvalue of $\widehat{\mathcal{A}}_{(\ell)}$ is positive (nonnegative), for each $\ell\in [p]$.
\end{theorem}
\begin{proof}
    This result immediately follows from Fact \ref{Degree k t-Hermitian form as a collection of degree k Hermitian forms}, Lemma \ref{Characterizing t-Hermitian Positivity}, and the Spectral Theorem of Hermitian Forms (i.e. Theorem \ref{Spectral Theorem of Hermitian Forms}). 
\end{proof}
\begin{example}
    Consider the $k=1$ case: positive definite (semi-definite) degree $1$ Hermitian forms. In \cite{zheng2021t}, M. Zheng et al. define the notion of t-positive definite for $t$-symmetric order $3$ hypermatrices $\mathcal{A}\in \mathbb{R}^{n\times n\times p}$, which recast over $\mathbb{C}$ states that: a $t$-Hermitian hypermatrix $\mathcal{A}\in \mathbb{C}^{n\times n\times p}$ is $t$-Hermitian positive definite (semi-definite) if 
    \[\langle \mathcal{Z}, \mathcal{A}*_t \mathcal{Z}\rangle > 0\text{ }(\geq 0)\]
    for every nonzero $\mathcal{Z}\in \mathbb{C}^{n\times 1\times p}$. In \cite[Theorem 4.4]{zheng2021t}, Zheng et al. then prove a real-version of the following result: a $t$-Hermitian hypermatrix $\mathcal{A}\in \mathbb{C}^{n\times n\times p}$ is $t$-Hermitian positive definite (semi-definite) if and only if each frontal slice of $\mathcal{A}$ in the Fourier domain, $\widehat{\mathcal{A}}_{(\ell)}$, is Hermitian positive definite (semi-definite) for each $\ell\in [p]$, which is true if and only if every eigenvalue of $\widehat{\mathcal{A}}_{(\ell)}$ is positive (nonnegative), for each $\ell\in [p]$. Thus, our Theorem \ref{Higher order t-Hermitian positive definite characterization} is not only an extension of Hermitian positivity of degree $k$ forms to degree $k$ $t$-Hermitian forms, it is also a higher degree complex analogue of Zheng et al.'s characterization of $t$-positive definiteness (semi-definiteness) of $t$-quadratic forms. 
    
    Indeed, since the normalized Discrete Fourier Transform matrix $\mathrm{DFT}_p$ is unitary \cite{landsberg2024quantum}, for any tubes $a,b\in \mathbb{T}_p$, which we may view as column vectors, we have that 
    \[\langle a,b\rangle = \langle \mathrm{DFT}_pa,\mathrm{DFT}_pb\rangle = \frac{1}{p}\langle \widehat{a},\widehat{b}\rangle,\]
    hence it follows that for any $\mathcal{Z},\mathcal{W}\in \mathbb{T}_p^n$, 
    \[\langle \mathcal{Z},\mathcal{W}\rangle = \frac{1}{p}\langle \widehat{\mathcal{Z}},\widehat{\mathcal{W}}\rangle = \sum\limits_{\ell=1}^p(z^{\ell})^{\dagger}w^{\ell}\]
    with $z^{\ell} := \widehat{\mathcal{Z}}_{(\ell)}$ and $w^{\ell} := \widehat{\mathcal{W}}_{\ell}$. Consequently, from Fact \ref{t-product in frequency domain slicewise reduces to matrix multiplication} it follows that 
    \[\langle \mathcal{Z},\mathcal{A}*_t\mathcal{Z}\rangle = \frac{1}{p}\sum\limits_{\ell=1}^p(z^{\ell})^{\dagger}\widehat{\mathcal{A}}_{(\ell)}z^{\ell} = \frac{1}{p}\sum\limits_{\ell=1}^p\widehat{h_{\mathcal{A}}(\mathcal{Z})}[\ell],\]
    with the second equality following from Example \ref{Degree 1 t-Hermitian Forms}. Assuming that $h_{\mathcal{A}}$ is $t$-Hermitian positive definite, then $\widehat{h_{\mathcal{A}}(\mathcal{Z})}\in \mathbb{T}_P^+\setminus \{0\}$, which means that $\widehat{h_{\mathcal{A}}(\mathcal{Z})}[\ell]\geq 0$ for each $\ell\in [p]$ and is in fact strictly positive for atleast one such $\ell\in [p]$; hence it follows that $\langle \mathcal{Z},\mathcal{A}*_t\mathcal{Z}\rangle > 0$ for every nonzero $\mathcal{Z}\in \mathbb{T}_p^n$. Conversely, suppose 
    \[\langle Z,\mathcal{A}*_t\mathcal{Z}\rangle > 0\]
    for all nonzero $\mathcal{Z}\in \mathbb{T}_p^n\cong \mathbb{C}^{n\times 1\times p}$. Fix arbitrary $\ell_*\in [p]$ and $z\in \mathbb{C}^n$, and let $\mathcal{Z}(\ell_*,z)$ denote the unique complex tubal vector such that 
    \[\widehat{\mathcal{Z}(\ell_*,z)}_{(\ell)} = \begin{cases}
        \begin{rcases}
            z, & \ell = \ell_* \\
            0, & \ell\neq \ell_*
        \end{rcases}.
    \end{cases}\]
    Then 
    \[\langle \mathcal{Z}(\ell_*,z),\mathcal{A}*_t\mathcal{Z}(\ell_*,z)\rangle = \frac{1}{p}z^{\dagger}\widehat{\mathcal{A}}_{(\ell_*)}z,\]
    which by assumption is positive, hence $z^{\dagger}\widehat{\mathcal{A}}_{\ell_*}z > 0$. Since $z\in \mathbb{C}^n$ was an arbitrary nonzero complex vector, it follows that $\widehat{\mathcal{A}}_{(\ell_*)}$ is Hermitian positive definite; additionally, since $\ell_*\in [p]$ was arbitrary, in fact it follows that $\widehat{\mathcal{A}}_{(\ell)}$ is Hermitian positive definite for each $\ell\in [p]$. Thus, by the Spectral Theorem for $t$-Hermitian forms, $h_{\mathcal{A}}$ is $t$-Hermitian positive definite. 
\end{example}
From the Spectral Theorem of $t$-Hermitian forms, Fact \ref{Sufficient Condition for Hermitian Positive Semidefiniteness}, Theorem \ref{Characterizing Positive Definiteness of Symmetric Matricizations}, and Corollary \ref{Sufficient Condition for Hermitian Positive Definiteness}, we immediately obtain the following result. 
\begin{corollary}[Sufficient Condtions for $t$-Hermitian Non-negativity/Positivity]\label{Sufficient Condtions for t-Hermitian Non-negativity/Positivity}
    Let $\mathcal{A}\in \mathbb{C}^{n\times\dots\times n\times p}$ be an order $2k+1$ $t$-conjugate partially symmetric hypermatrix, $h_{\mathcal{A}}$ be its corresponding $t$-Hermitian form, and denote $ M_{\ell} := M_{cb}(\widehat{\mathcal{A}}_{(\ell)})$ and $S_{\ell} := M_{sym}(\widehat{\mathcal{A}}_{(\ell)})$ for each $\ell\in [p]$. Then, 
    \begin{itemize}
        \item if for each $\ell\in [p]$, either $M_{\ell}$ or $S_{\ell}$ is Hermitian positive semidefinite, then $h_{\mathcal{A}}$ is $t$-Hermitian positive semidefinite;  
        \item if for each $\ell\in [p]$, $S_{\ell}$ is Hermitian positive definite, then $h_{\mathcal{A}}$ is $t$-Hermitian positive definite. 
    \end{itemize}
\end{corollary}

\section{Positivity-Preserving Contractions}\label{Application to Classical Hermitian Forms}
\subsection{Contraction-Induced Hermitian Positive Forms}
\begin{defn}
    Let $c\in \mathbb{T}_p$ and $\mathcal{A}\in \mathbb{C}^{n\times\dots\times n\times p}$ be an order $2k+1$ hypermatrix. We call the order $2k$ hypermatrix
    \[\Phi_c(\mathcal{A}) := \sum\limits_{\ell=1}^p\overline{c[\ell]}\mathcal{A}_{(\ell)}\]
    the \textbf{$\boldsymbol{c}$-contraction hypermatrix} of $\mathcal{A}$. 
\end{defn}
\begin{remark}
    Note that by the inverse fast Fourier transformation, for each $\ell\in [p]$ we have that 
    \[\mathcal{A}_{(\ell)} = \frac{1}{p}\sum\limits_{m=1}^p\omega_p^{-(\ell-1)(m-1)}\widehat{\mathcal{A}}_{(m)},\]
    therefore
    \begin{equation}\label{contraction hypermatrix in frequency domain}
        \Phi_c(\mathcal{A}) = \sum\limits_{\ell=1}^p\overline{c[\ell]}\mathcal{A}_{(\ell)} = \sum\limits_{\ell,m=1}^p\overline{c[\ell]}\omega_p^{-(\ell-1)(m-1)}\widehat{\mathcal{A}}_{(m)} = \frac{1}{p}\sum\limits_{m=1}^p\overline{\widehat{c}[m]}\widehat{\mathcal{A}}_{(m)}.
    \end{equation}
\end{remark}
\begin{proposition}[Quantitative Lower Bound for Contraction Induced Forms]\label{Quantitative Lower Bound for Contraction Induced Forms}
    Let $\mathcal{A}\in \mathbb{C}^{n\times\dots\times n\times p}$ be an order $2k+1$ $t$-Hermitian positive semidefinite hypermatrix and $c\in \mathbb{T}_p^+$, and define 
    \[\alpha_{\ell} := \min\limits_{\|z\|=1}h_{\widehat{\mathcal{A}}_{(\ell)}}(z).\]
    Then \[h_{\Phi_c(\mathcal{A})}(z) \geq \gamma_c(\mathcal{A})\|z\|^{2k}\qquad \forall z\in \mathbb{C}^n,\]
    where 
    \[\gamma_c(\mathcal{A}) := \frac{1}{p}\sum\limits_{l=1}^p\widehat{c}[\ell]\alpha_{\ell},\]
    hence $h_{\Phi_c(\mathcal{A})}(x)$ is Hermitian positive semidefinite. 
    
    Furthermore, if additionally $\mathcal{A}$ is $t$-Hermitian positive definite and $c\neq 0$, then $\gamma_c(\mathcal{A}) > 0$, hence $h_{\Phi_c(\mathcal{A})}(x)$ is Hermitian positive definite. 
\end{proposition}
\begin{proof}
    First note that $\alpha_{\ell}$ exists because the unit sphere is compact and Hermitian forms are continuous, hence the minimum is obtained. Furthermore, by $t$-Hermitian positive semidefiniteness of $\mathcal{A}$, 
    \[h_{\widehat{\mathcal{A}}_{(\ell)}}(z) \geq 0\]
    for all nonzero $z\in \mathbb{C}^n$ and each $\ell\in [p]$, hence $\alpha_{\ell}\geq 0$ for each $\ell\in [p]$, in which case it follows that $\gamma_c(\mathcal{A})\geq 0$ since we are assuming $c\in \mathbb{T}_p^+$. Additionally, since $z\neq 0$, by homogeneity of the Hermitian form (which recall is of bi-degree $(k,k)$), we have that 
    \[h_{\widehat{\mathcal{A}}_{(\ell)}}(z) = \|z\|^{2k}h_{\widehat{\mathcal{A}}_{(\ell)}}\left(\frac{z}{\|z\|}\right) \geq \alpha_{\ell}\|z\|^{2k}.\]
    Therefore,  
    \begin{align*}
        h_{\Phi_c(\mathcal{A})}(z) &= \frac{1}{p}\sum\limits_{\ell=1}^p\widehat{c}[\ell]h_{\widehat{\mathcal{A}}_{(\ell)}}(z),\quad \text{by equation }\eqref{contraction hypermatrix in frequency domain}\text{ and since }c\in \mathbb{T}_p^+ \\
        &\geq \frac{1}{p}\sum\limits_{\ell=1}^p\widehat{c}[\ell]\alpha_{\ell}\|z\|^{2k},\quad \text{by above} \\
        &= \gamma_c(\mathcal{A})\|z\|^{2k},\quad \text{by definition},
    \end{align*}
    and so consequently we immediately have that $h_{\Phi_c(\mathcal{A})}(x)$ is Hermitian positive semidefinite. 
    
    Now, if additionally we assume that $\mathcal{A}$ is $t$-Hermitian positive definite and that $c\neq 0$, then 
    \[h_{\widehat{\mathcal{A}}_{(\ell)}}(z) > 0\]
    for all nonzero $z\in \mathbb{C}^n$ and each $\ell\in [p]$, and so by continuity of $h_{\widehat{\mathcal{A}}_{(\ell)}}$ it follows that $\alpha_{\ell} > 0$ for each $\ell\in [p]$. Moreover, since $c\neq 0$ there exists atleast one $\ell\in [p]$ such that $\widehat{c}[\ell] > 0$, hence $\gamma_c(\mathcal{A}) > 0$. Thus, from the calculation above, 
    \[h_{\Phi_c(\mathcal{A})}(z) \geq \gamma_c(\mathcal{A})\|z\|^{2k} > 0\]
    for all nonzero $z\in \mathbb{C}^n$, from which it immediately follows that $h_{\Phi_c(\mathcal{A})}(x)$ is Hermitian positive definite. 
\end{proof}

\subsection{Characterizing Positivity-Preserving Contractions}
Below we characterize when $c$-contractions preserve positivity. In particular, we prove that if $\mathcal{A}$ is $t$-Hermitian positive definite, then $\Phi_c(\mathcal{A})$ induces a Hermitian form that is positive on $\mathbb{C}^n\setminus \{0\}$ if and only if $c$ belongs to the cone of positive tubes. While the forward implication is unsurprising as it immediately follows from equation \eqref{contraction hypermatrix in frequency domain}, the backwards implication is less immediate and has interesting geometric consequences. In particular, the backwards implication is equivalent to saying that if $c$ is not in the positive cone, then $\Phi_c$ will not universally preserve positivity; that is, atleast hypothetically, we should always be able to find a $t$-Hermitian positive definite hypermatrix such that $\Phi_c(\mathcal{A})$ does not correspond to a positive Hermitian form. Indeed, after stating and proving the characterization theorem, we will give an explicit example of such phenomena, showing that no matter how close a tube is to the positive cone, if it is not in the positive cone, then positivity is eventually violated. 

Without further ado, we now state and prove the characterization theorem. 
\begin{theorem}[Characterization of Positivity-Preserving Contractions]
    Let $c\in \mathbb{T}_p$. The following are equivalent:
    \begin{itemize}
        \item[(i)] $c\in \mathbb{T}_p^+$; that is, $\widehat{c}[\ell]\geq 0$ for each $\ell\in [p]$. 
        \item[(ii)] $\Phi_c$ \textbf{weakly preserves positivity}, meaning that for every order $2k+1$ $t$-Hermitian positive semidefinite hypermatrix $\mathcal{A}\in \mathbb{C}^{n\times\dots\times n\times p}$, $h_{\Phi_c(\mathcal{A})}(x)$ is Hermitian positive semidefinite. 
    \end{itemize}
    Additionally, the following are equivalent:
    \begin{itemize}
        \item[(a)] $c\in \mathbb{T}_p^+\setminus \{0\}$; that is, $\widehat{c}[\ell]\geq 0$ for all $\ell\in [p]$ and $\widehat{c}[\ell_*] > 0$ for atleast one $\ell_*\in [p]$. 
        \item[(b)] $\Phi_c$ \textbf{strongly preserves positivity}, meaning that for every order $2k+1$ $t$-Hermitian positive definite hypermatrix $\mathcal{A}\in \mathbb{C}^{n\times\dots\times n\times p}$, $h_{\Phi_c(\mathcal{A})}(x)$ is Hermitian positive definite. 
    \end{itemize}
\end{theorem}
\begin{proof}
    Both equation \eqref{contraction hypermatrix in frequency domain} and Proposition \ref{Quantitative Lower Bound for Contraction Induced Forms} immediately give $(i)\Rightarrow (ii)$. 

    Conversely, suppose that $\Phi_c$ weakly preserves positivity; that is, for every order $2k+1$ $t$-Hermitian positive semidefinite hypermatrix $\mathcal{A}\in \mathbb{C}^{n\times\dots\times n\times p}$, the $c$-contraction hypermatrix
    \[\Phi_c(\mathcal{A}) = \sum\limits_{\ell=1}^p\overline{c[\ell]}\mathcal{A}_{(\ell)}\]
    represents a classical Hermitian positive semidefinite form. Fix $m\in [p]$, let $\mathcal{B}\in\mathbb{C}^{n\times\dots\times n}$ be any nonzero order $2k$ Hermitian positive semidefinite hypermatrix, and define the order $2k+1$ hypermatrix $\mathcal{A}\in \mathbb{C}^{n\times\dots\times n\times p}$ so that 
    \[\widehat{\mathcal{A}}_{(m)} = \mathcal{B}\]
    and 
    \[\widehat{\mathcal{A}}_{(\ell)} = 0\]
    for all $\ell\neq m$. By construction, every frontal slice of $\mathcal{A}$ in the frequency domain represents a classical Hermitian positive semidefinite form, hence $\mathcal{A}$ represents a $t$-Hermitian positive semidefinite form. Furthermore, by equation \eqref{contraction hypermatrix in frequency domain}, we have that  
    \[\Phi_c(\mathcal{A}) = \frac{1}{p}\overline{\widehat{c}[m]}\mathcal{B},\]
    thus the $c$-contraction hypermatrix $\Phi_c(\mathcal{A})$ represents the classical Hermitian positive semidefinite form
    \[h_{\Phi_c(\mathcal{A})}(x) = \frac{1}{p}\overline{\widehat{c}[m]}h_{\mathcal{B}}(x).\]
    Moreover, since $\mathcal{B}\neq 0$ and is Hermitian positive semidefinite, this implies that there exists $z_*\in \mathbb{C}^n$ such that 
    \[h_{\mathcal{B}}(z_*)>0.\]
    Therefore, Hermitian positive semidefiniteness of $\frac{1}{p}\overline{\widehat{c}[m]}h_{\mathcal{B}}(x)$ implies that 
    \[\frac{1}{p}\overline{\widehat{c}_m}h_{\mathcal{B}}(z_*)\geq 0,\]
    hence $\overline{\widehat{c}[m]}\geq 0$; that is, $\widehat{c}[m]\geq 0$. Since $m$ was arbitrary, repeatedly applying this construction yields $\widehat{c}[\ell]\geq 0$ for all $l\in [p]$. Thus, $(ii)\Rightarrow (i)$. 

    Next, suppose $c\in \mathbb{T}_p^+\setminus \{0\}$ and let $\mathcal{A}\in \mathbb{C}^{n\times\dots\times n\times p}$ be any order $2k+1$ $t$-Hermitian positive definite hypermatrix. By $t$-Hermitian positive definiteness, for every nonzero $z\in \mathbb{C}^n$ we have that 
    \[h_{\widehat{\mathcal{A}}_{(\ell)}}(z) > 0\]
    for each $\ell\in [p]$. Consequently, 
    \[h_{\Phi_c(\mathcal{A})}(z) = \frac{1}{p}\sum\limits_{\ell=1}^p\widehat{c}[\ell]h_{\widehat{\mathcal{A}}_{(\ell)}}(z)>0\]
    for all nonzero $z\in \mathbb{C}^n$, proving that classical Hermitian form induced by $\Phi_c(\mathcal{A})$ is Hermitian positive definite. Thus, $(a)\Rightarrow (b)$. Conversely, suppose $\Phi_c$ strongly preserves positivity; that is, for every order $2k+1$ $t$-Hermitian positive definite hypermatrix $\mathcal{A}\in \mathbb{C}^{n\times\dots\times n\times p}$, $h_{\Phi_c(\mathcal{A})}(x)$ is Hermitian positive semidefinite. For any order $2k+1$ $t$-Hermitian positive definite hypermatrix $\mathcal{A}\in \mathbb{C}^{n\times\dots\times n\times p}$, let $\epsilon > 0$ and $\mathcal{B}\in \mathbb{C}^{n\times\dots\times n\times p}$ be an order $2k+1$ $t$-Hermitian positive definite hypermatrix. Then $\mathcal{A}+\epsilon\mathcal{B}$ is $t$-Hermitian positive definite, and by assumption 
    \[\Phi_c(\mathcal{A}+\epsilon\mathcal{B}) = \Phi_c(\mathcal{A})+\epsilon\Phi_c(\mathcal{B})\]
    represents a strictly positive Hermitian form. That is, for any nonzero $z\in \mathbb{C}^n$, 
    \[h_{\Phi_c(\mathcal{A})}(z)+\epsilon h_{\Phi_c(\mathcal{B})}(z) > 0,\]
    and so taking $\epsilon\rightarrow 0^+$, it follows that 
    \[h_{\Phi_c(\mathcal{A})}(z)\geq 0.\]
    Hence, $\Phi_c(\mathcal{A})$ is Hermitian positive semidefinite, showing that $\Phi_c$ also weakly preserves positivity. From the previous paragraph, this implies that $c\in \mathbb{T}_p^+$. Indeed, it must be the case that $c\in \mathbb{T}_p^+\setminus \{0\}$, for otherwise if $c=0$, then $\Phi_c(\mathcal{A})=0$ for all $\mathcal{A}\in \mathbb{C}^{n\times\dots\times n\times p}$, which does not represent a Hermitian positive definite form, violating the assumption that $\Phi_c$ strongly preserves positivity. Thus, $(b)\Rightarrow (a)$. 
\end{proof}
\begin{remark}
    In fact, if $c\in \mathbb{T}_p^+\setminus \{0\}$, then in certain instances we may relax the strict $t$-Hermitian positive definiteness assumption of $\mathcal{A}$ to $t$-Hermitian positive semidefiniteness and still guarantee $\Phi_c(\mathcal{A})$ to be strictly Hermitian positive definite. 

    In particular, let $\mathcal{A}\in \mathbb{C}^{n\times\dots\times n\times p}$ be an order $2k+1$ $t$-Hermitian positive semidefinite hypermatrix, $c\in \mathbb{T}_p^+\setminus \{0\}$, and define 
    \[I_c := \{\ell\in [p]:\widehat{c}[\ell]>0\}.\]
    Then we claim that $\Phi_c(\mathcal{A})$ is Hermitian positive definite if and only if 
    \[\bigcap\limits_{\ell\in I_c}\left\{z\in \mathbb{C}^n:h_{\widehat{\mathcal{A}}_{(\ell)}}(z)=0\right\} = \{0\}.\]
    Indeed by \eqref{contraction hypermatrix in frequency domain}, $\Phi_c(\mathcal{A}) = \frac{1}{p}\sum\limits_{l\in I_c}\widehat{c}[\ell]\widehat{\mathcal{A}}_{(l)}$, and by the Characterization of Positivity-Preserving Contractions Theorem, this uniquely represents the classical Hermitian positive semidefinite form
    \[h_{\Phi_c(\mathcal{A})}(x) = \frac{1}{p}\sum\limits_{l\in I_c}\overline{\widehat{c}[\ell]}h_{\widehat{\mathcal{A}}_{(\ell)}}(x).\]
    This form vanishes on precisely $\bigcap\limits_{\ell\in I_c}\left\{z\in \mathbb{C}^n:h_{\widehat{\mathcal{A}}_{(\ell)}}(z)=0\right\}$, hence $h_{\Phi_c(\mathcal{A})}$ is strictly positive if and only if $\bigcap\limits_{\ell\in I_c}\left\{z\in \mathbb{C}^n:h_{\widehat{\mathcal{A}}_{(\ell)}}(z)=0\right\} = \{0\}$.  
\end{remark}
Below is a useful consequence of the Characterization of Positivity-Preserving Contractions Theorem, which we will later utilize in an example. 
\begin{corollary}\label{Symmetric Tubes and Preserving Positivity}
    Let $c$ be a real-valued $p$-dimensional tube, $c[\ell]\leq 0$ for each $\ell=2,\dots,p$, and suppose that $c$ is "symmetric" in the sense that 
    \[c[\ell] = c[p-\ell+2]\]
    for each $\ell=2,\dots,p$. Then $\Phi_c$ preserves positivity if and only if 
    \[c[1] \geq \sum\limits_{m=2}^p\left|c[m]\right|.\]
\end{corollary}
\begin{proof}
    Let $b$ be $c$ re-indexed starting at $0$, so that $b[\ell]  = c[\ell+1]$ for $\ell=0,1,\dots,p-1$. Then the symmetry condition
    \[c[\ell] = c[p-\ell+2]\]
    with $\ell=2,\dots,p$ is equivalent to 
    \[b[\ell] = b[-\ell\bmod p]\]
    with $\ell=0,1,\dots,p-1$. Applying the FFT to $b$, for each $\ell=0,1,\dots,p-1$ we have that  
    \begin{align*}
        \widehat{b}[\ell] &= \sum\limits_{m=0}^{p-1}b[m]e^{-\frac{2\pi i \ell m}{p}} \\
        &= b[0] + \sum\limits_{m=1}^{p-1}b[-m\bmod p]e^{-\frac{2\pi i \ell m}{p}} \\
        &= b[0] + \sum\limits_{m'=-(p-1)}^{-1}b[m'\bmod p]e^{\frac{2\pi i \ell m'}{p}},\quad \text{setting }m':=-m \\
        &= b[0] + \sum\limits_{m'=1}^{p-1}b[m']e^{\frac{2\pi i \ell m'}{p}},\quad \text{by cyclicity} \\
        &= \overline{\sum\limits_{m'=0}^{p-1}b[m']e^{-\frac{2\pi i \ell m'}{p}}} \\
        &= \overline{\widehat{b}[\ell]},
    \end{align*}
    proving that $\widehat{b}$, and hence $\widehat{c}$, is real-valued. Consequently, by Euler's formula, for each $\ell\in [p]$ we have that 
    \[\widehat{c}[\ell] = c[1] - \sum\limits_{m=2}^p|c[m]|\cos\left(\frac{2\pi (\ell-1)(m-1)}{p}\right) \geq c[1]-\sum\limits_{m=2}^p|c[m]|.\]
    Thus, if $c[1] \geq \sum\limits_{m=2}^p|c[m]|$, then $c$ is Fourier-positive, in which case $\Phi_c$ preserves positivity by the Characterization of Positivity-Preserving Contractions Theorem. Conversely, if $c[1] < \sum\limits_{m=2}^p|c[m]|$, then 
    \[\widehat{c}[1] = c[1] - \sum\limits_{m=2}^p|c[m]| < 0,\]
    implying that $c$ is not Fourier-positive, hence $\Phi_c$ does not preserve positivity by the Characterization of Positivity-Preserving Contractions Theorem. 
\end{proof}

\begin{example}\label{Contraction Induced Quartic Example}
    Let $x = \left[\begin{array}{c}
        x_1 \\
        x_2
    \end{array}\right]$ be a vector of complex variables and let $v(x) = \left[\begin{array}{c}
        x_1^2 \\
        \sqrt{2}x_1x_2 \\
        x_2^2
    \end{array}\right]$. Note that $v(x)$ is the formal matrix product $U^Tx^{\otimes 2}$ with $U$ the $4\times 3$ matrix representation of the isometry $\mathrm{Sym}^2(\mathbb{C}^2)\hookrightarrow \mathbb{C}^4$. Now, consider  
    \[S_1 = \left[\begin{array}{ccc}
        3 & 0 & 0 \\
        0 & -2 & 0 \\
        0 & 0 & 3
    \end{array}\right].\]
    $S_1$ is indefinite, however its corresponding Hermitian quartic form
    \begin{align*}
        v(x)^{\dagger}S_1v(x) &= 3|x_1|^4-4|x_1|^2|x_2|^2+3|x_2|^4 \\
        &= \frac{1}{2}\big(|x_1|^2+|x_2|^2\big)^2+\frac{5}{2}\big(|x_1|^2-|x_2|^2\big)^2 \\
        &\geq \frac{1}{2}\|x\|^4
    \end{align*}
    is Hermitian positive definite. For $z\in\mathbb{C}$, define 
    \[Q_z := \left[\begin{array}{cc}
        1 & z \\
        0 & 1
    \end{array}\right]\]
    and 
    \[P_z := \left[\begin{array}{ccc}
        1 & \sqrt{2}z & z^2 \\
        0 & 1 & \sqrt{2}z \\
        0 & 0 & 1
    \end{array}\right];\]
    note in particular that $v(Qx) = v(x_1+zx_2,x_2) = P_zv(x)$. Now define   
    \[S_2 := P_1^{\dagger}S_1P_1\quad\text{and}\quad S_3 := P_{1+i}^{\dagger}S_1P_{1+i};\]
    explicitly, 
    \[S_2 = \left[\begin{array}{ccc}
        3 & 3\sqrt{2} & 3 \\
        3\sqrt{2} & 4 & \sqrt{2} \\
        3 & \sqrt{2} & 2
    \end{array}\right]\]
    and 
    \[S_3 = \left[\begin{array}{ccc}
        3 & 3\sqrt{2}(1+i) & 6i \\
        3\sqrt{2}(1-i) & 10 & 4\sqrt{2}(1+i) \\
        -6i & 4\sqrt{2}(1-i) & 7
    \end{array}\right].\]
    Since $S_2$ and $S_3$ are by definition congruent to $S_1$, by Sylvester's Law of Inertia each of these matrices are indefinite. Nevertheless,  
    \[v(x)^{\dagger}S_2v(x) = v(Q_1x)^{\dagger}S_1v(Q_1x)\]
    and 
    \[v(x)^{\dagger}S_3v(x) = v(Q_{1+i}x)^{\dagger}S_1v(Q_{1+i}x);\]
    that is, the Hermitian quartic forms corresponding to $S_2$ and $S_3$ are obtained by an invertible change of variables transformation of the Hermitian positive definite form corresponding to $S_1$, hence they are Hermitian positive definite too.  

    For each $\lambda > 0$, let $\mathcal{A}_{\lambda}\in \mathbb{C}^{2\times 2\times 2\times 2\times 3}$ be such that 
    \[\begin{cases}
        \begin{rcases}
            M_{sym}(\widehat{\mathcal{A}}_{(1)}) = \lambda S_1 \\
            M_{sym}(\widehat{\mathcal{A}}_{(\ell)}) = S_{\ell}, & \ell = 2,3
        \end{rcases}. 
    \end{cases}\]
    Since each frontal slice of $\mathcal{A}_{\lambda}$ in the frequency domain corresponds to a Hermitian positive definite quartic form, it follows that $\mathcal{A}_{\lambda}$ is $t$-Hermitian positive definite for all $\lambda > 0$. Now consider the tube 
    \[c_+ := (2,-1,-1).\]
    Since $2 \geq |-1|+|-1| = 2$, by Corollary \ref{Symmetric Tubes and Preserving Positivity} it follows that $c_+$ is Fourier positive and $\Phi_{c_+}(\mathcal{A}_{\lambda})$ represents a classical Hermitian positive definite quartic form for all $\lambda > 0$. 
    
    In particular, applying the FFT to $c_+$, we obtain $\widehat{c_+} = (0,3,3)$, and by equation \eqref{contraction hypermatrix in frequency domain} we have that 
    \[M_{sym}\big(\Phi_{c_+}(\mathcal{A}_{\lambda})\big) = \frac{1}{3}\big(0\cdot\lambda S_1 + 3S_2 + 3S_3\big) = S_2+S_3,\]
    annihilating the $\lambda$-dependent component. Thus, for all $\lambda > 0$, $\Phi_{c_+}(\mathcal{A}_{\lambda})$ represents the same Hermitian positive definite quartic form obtained from the matrix 
    \[S_2 + S_3 = \left[\begin{array}{ccc}
        6 & 3\sqrt{2}(2+i) & 3+6i \\
        3\sqrt{2}(2-i) & 14 & \sqrt{2}(5+4i) \\
        3-6i & \sqrt{2}(5-4i) & 9
    \end{array}\right].\]
    Clearly this matrix corresponds to a Hermitian positive definite quartic form since it is the sum of two matrices each representing Hermitian positive definite quartic forms. However, this fact is obscured by only looking at the normalized symmetric matrix representation of $h_{\Phi_{c_+}(\mathcal{A}_{\lambda})}(x)$ or the form itself. Indeed, note that the determinant of this matrix is $-24$, hence the normalized symmetric matrix representation of $h_{\Phi_{c_+}(\mathcal{A}_{\lambda})}(x)$ is indefinite. Furthermore, when writing out the form explicitly, we obtain 
    \begin{align*}
        h_{\Phi_{c_+}(\mathcal{A}_{\lambda})}(x) &= v(x)^{\dagger}M_{sym}\big(\Phi_c(\mathcal{A})\big)v(x) \\
        &= 6|x_1|^4+28|x_1|^2|x_2|^2+9|x_2|^4 \\
        &\quad + 12\Re{(2+i)|x_1|^2\overline{x_1}x_2} \\
        &\quad + 6\Re{(1+2i)\overline{x_1}^2x_2^2} \\
        &\quad + 4\Re{(5+4i)|x_2|^2\overline{x_1}x_2},
    \end{align*}
    which is not obviously positive. In particular, the latter 3 terms may be negative and the magnitude of their coefficients adds up to 
    \[18\sqrt{5}+4\sqrt{41}\approx 65.862,\]
    whereas the coefficients for the first $3$ necessarily positive terms only adds up to 
    \[6+28+9 = 43,\]
    so the necessarily positive terms do not clearly dominate the potentially negative terms. As it turns out, the deeper reason why we are guaranteed positivity is because the tube $c_+$ is Fourier positive. In fact, the tube $c_+$ is on the boundary of $\mathbb{T}_3^+$ (a closed set), yet by the Characterizing Positivity-Preserving Contractions Theorem, any tube $c$ outside of $\mathbb{T}_3^+$ does not yield a contraction $\Phi_c$ which preserves positivity; indeed, in this example in particular, we will show that no matter how close such a tube $c$ is to the boundary of $\mathbb{T}_3^+$, if it is not in $\mathbb{T}_3^+$, then there exists $\lambda > 0$ such that $\Phi_c(\mathcal{A}_{\lambda})$ does not represent a Hermitian positive form. 
    
    We demonstrate the sharpness of this condition by considering the tube 
    \[c_{\epsilon} = (2-\epsilon,-1,-1),\]
    with $\epsilon > 0$ arbitrary. Since $2-\epsilon \ngeq |-1|+|-1| = 2$, $c_{\epsilon}$ is not Fourier positive, but note that it can be taken to be arbitrarily close to $\mathbb{T}_3^+$ by taking $\epsilon$ to be arbitrarily small. By applying the FFT to $c_{\epsilon}$, we obtain $\widehat{c_{\epsilon}} = (-\epsilon,3-\epsilon,3-\epsilon)$, and so by equation \eqref{contraction hypermatrix in frequency domain} we have that 
    \[M_{sym}\big(\Phi_{c_{\epsilon}}(\mathcal{A}_{\lambda})\big) = \frac{1}{3}\big(-\lambda\epsilon S_1 + (3-\epsilon)S_2 + (3-\epsilon)S_3\big).\]
    If $e_1$ denotes the basis vector $\left[\begin{array}{c}
        1 \\
        0 
    \end{array}\right]$, then $v(e_1) = \left[\begin{array}{c}
        1 \\
        0 \\
        0
    \end{array}\right]$, hence it follows that 
    \[h_{\Phi_{c_{\epsilon}}(\mathcal{A}_{\lambda})}(e_1) = \frac{1}{3}\big(-\lambda\epsilon\underbrace{S_1[1,1]}_{3} + (3-\epsilon)\underbrace{S_2[1,1]}_{3} + (3-\epsilon)\underbrace{S_3[1,1]}_{3}\big) = 6 - (2+\lambda)\epsilon\]
    which is strictly less than $0$ whenever $\frac{6}{\epsilon}-2 < \lambda$. Therefore, although $\mathcal{A}_{\lambda}$ is $t$-Hermitian positive definite for every $\lambda > 0$ and $c_{\epsilon}$ may be taken to be arbitrarily close to the positive cone $\mathbb{T}_3^+$, the fact that $c_{\epsilon}\notin \mathbb{T}_3^+$ for each $\epsilon > 0$ implies that positivity is eventually violated. More precisely, for every $\epsilon > 0$, choosing $\lambda$ such that $\lambda > \frac{6}{\epsilon}-2$ produces a $t$-Hermitian positive definite hypermatrix $\mathcal{A}_{\lambda}$ for which the contraction $\Phi_{c_{\epsilon}}(\mathcal{A}_{\lambda})$ is not Hermitian positive definite. Thus, the positive cone $\mathbb{T}_3^+$ is the exact region for which tubes yield positivity-preserving contractions.  

    Now, back to positivity-preserving contraction $\Phi_{c_+}(\mathcal{A}_{\lambda})$. By Proposition \ref{Quantitative Lower Bound for Contraction Induced Forms}, 
    \begin{equation}\label{lower bound}
        h_{\Phi_c(\mathcal{A})}(x) \geq \frac{1}{3}(0\alpha_1+3\alpha_2+3\alpha_3)\|x\|^4 = (\alpha_2+\alpha_3)\|x\|^4,
    \end{equation}
    where recall \[\alpha_{\ell} := \min\limits_{\|z\|=1}h_{\widehat{\mathcal{A}}_{(\ell)}}(x) = \min\limits_{\|z\|=1}v(x)^{\dagger}S_{\ell}v(x)\]
    for each $\ell\in [3]$. Since
    \begin{align*}
        v(x)^{\dagger}S_1v(x) &= 3|x_1|^4-4|x_1|^2|x_2|^2+3|x_2|^4 \\
        &= \frac{1}{2}\big(|x_1|^2+|x_2|^2\big)^2+\frac{5}{2}\big(|x_1|^2-|x_2|^2\big)^2 \\
        &\geq \frac{1}{2}\|x\|^4,
    \end{align*}
    it follows that 
    \begin{align*}
        v(x)^{\dagger}S_2v(x) &= v(Q_1x)^{\dagger}S_1v(Q_1x) \\
        &\geq \frac{1}{2}\|Q_1x\|^4 \\
        &\geq \frac{\sigma_{\min}(Q_1)^4}{2}\|x\|^4
    \end{align*}
    and 
    \begin{align*}
        v(x)^{\dagger}S_3v(x) &= v(Q_{1+i}x)^{\dagger}S_1v(Q_{1+i}x) \\
        &\geq \frac{1}{2}\|Q_{1+i}x\|^4 \\
        &\geq \frac{\sigma_{\min}(Q_{1+i})^4}{2}\|x\|^4,
    \end{align*}
    with $\sigma_{\min}$ denoting the minimum singular value.
    One can check that $\sigma_{\min}(Q_1)^2 = \frac{3-\sqrt{5}}{2}$ and $\sigma_{\min}(Q_{1+i})^2 = 2-\sqrt{3}$, therefore it follows that 
    \[\alpha_2 \geq \frac{1}{2}\left(\frac{3-\sqrt{5}}{2}\right)^2 = \frac{7-3\sqrt{5}}{4}\]
    and 
    \[\alpha_3 \geq \frac{1}{2}(2-\sqrt{3})^2 = \frac{7-4\sqrt{3}}{2}.\]
    Thus, by equation \eqref{lower bound}, we have that 
    \[h_{\Phi_{c_+}(\mathcal{A}_{\lambda})}(x) \geq \frac{21-3\sqrt{5}-8\sqrt{3}}{4}\|x\|^4.\]
    Note that $\frac{21-3\sqrt{5}-8\sqrt{3}}{4}\approx 0.1088$, hence we conclude that 
    \[h_{\Phi_{c_+}(\mathcal{A}_{\lambda})}(x) > 0.108\|x\|^4\]
    for all $\lambda > 0$. 
\end{example}

\section{Conclusion}
In this article, we introduced higher degree $t$-Hermitian forms which intrinsically are tube-valued analogues of classical degree $k$ Hermitian forms. We showed that these objects correspond uniquely to odd order hypermatrices whose frontal slices in the Fourier basis are conjugate partially symmetric, and from this fact it follows that $t$-Hermitian forms decompose into a Fourier-packaged collection of classical degree $k$ Hermitian forms. This decomposition yields spectral characterization of positivity in terms of the $\widehat{H}$-eigenvalues of the frontal slices of the corresponding hypermatrices in the Fourier basis, and as a consequence of this we obtain sufficient conditions for positivity of $t$-Hermitian forms in terms of normalized symmetric matricizations of said frontal slices. 

After developing this general theory of higher-degree $t$-Hermitian forms, we then considered tensor contractions along their tubal mode, which yield classical Hermitian forms, and characterize when they are positive definite/semidefinite as well as derive quantitative lower bounds for their positivity margins. We then demonstrate the sharpness of our characterization result and the the utility of the quantitative lower bounds with a concrete numerical example. 

Thus, the study of $t$-Hermitian forms not only provides a tubal extension of classical Hermitian form theory, but also a structured mechanism for producing and analyzing classical Hermitian positive forms through tensor contractions of tube-valued forms.

\appendix
\section{Proofs for Facts Regarding Symmetric Matricizations}
First we prove Fact \ref{Position in lexicographic order}, which recall states that the position of a tuple $\mathbf{i} = (i_1,\dots,i_k)\in \mathcal{I}_k$ is 
\[\varphi(\mathbf{i}) = 1+\sum\limits_{q=1}^k\binom{i_q+q-2}{q}.\]
\begin{proof}
    Let $\mathbf{i} = (i_1,\dots,i_k)\in \mathcal{I}_k$ and define $\widetilde{\mathbf{i}} := (n-i_k+1,\dots,n-i_1+1)$; note that the map $\mathbf{i}\mapsto \widetilde{\mathbf{i}}$ is an involution on $\mathcal{I}_k$. Moreover, $\mathbf{i} < \mathbf{j} \iff \widetilde{\mathbf{j}} < \widetilde{\mathbf{i}}$ with respect to the lexicographic order; hence, counting the number of tuples $\mathbf{j} = (j_1,\dots,j_k)\in \mathcal{I}_k$ such that $\mathbf{j} > \mathbf{i}$ is equivalent to counting all tuples in $\mathcal{I}_k$ that are less than $\widetilde{\mathbf{i}}$. 

    It is well known that the number of tuples $\mathbf{j} = (j_1,\dots ,j_k)$ such that $1\leq j_1\leq\dots\leq j_k\leq n$ is equal to $\binom{n+k-1}{k}$, hence it follows that the number of tuples $\mathbf{j} = (j_1,\dots,j_k)$ such that $m\leq j_1\leq\dots\leq j_k\leq n$ is equal to $\binom{n+k-m}{k}$. In the lexicographical order, $(j_1,\dots ,j_k) > (i_1,\dots, i_k)$ if and only if there exists $l\in [k]$ such that $j_l > i_l$ and $j_1= i_1,...,j_{l-1}=i_{l-1}$, or $j_1 > i_1$; in either case, the last $k-l+1$ entries of $(j_1,...,j_k)$ are subject only to the restriction that $i_l+1 \leq j_l\leq j_{l+1}\leq\dots\leq j_k\leq n$, from which it follows that there are exactly 
    \[\sum\limits_{l=1}^k\binom{n+(k-l+1)-(i_l+1)}{n-l+1} = \sum\limits_{l=1}^k\binom{n+k-l-i_l}{n-l+1}\] 
    many tuples $\mathbf{j} = (j_1,\dots,j_k)\in \mathcal{I}_k$ greater than $\mathbf{i} = (i_1,\dots ,i_k)$. 

    Now, since $\mathbf{i}\mapsto \widetilde{\mathbf{i}}$ is an involution on $\mathcal{I}_k$, there exists a unique $\mathbf{j}\in \mathcal{I}_k$ such that $\mathbf{i} = \widetilde{\mathbf{j}}$. Hence, there are exactly 
    \begin{align*}
        \sum\limits_{l=1}^k\binom{n+k-l-i_l}{k-l+1} &= \sum\limits_{l=1}^k\binom{n+k-l-(n-j_{k-l+1}+1)}{k-l+1} \\
        &= \sum\limits_{l=1}^k\binom{j_{k-l+1}+k-l-1}{k-l+1} \\
        &= \sum\limits_{l=1}^k\binom{j_{k-l+1}+(k-1+1)-2}{k-l+1} \\
        &= \sum\limits_{q=1}^k\binom{j_q+q-2}{q},\quad \text{setting }q:=k-l+1,
    \end{align*}
    tuples less than $\widetilde{\mathbf{j}}$. Consequently, the position of $\widetilde{\mathbf{j}} = \mathbf{i}$ is exactly $1+\sum\limits_{q=1}^k\binom{i_q+q-2}{q}$.  
\end{proof}
Next, we prove Fact \ref{Symmetric Matricization Embedding}, which recall stated that $M_{sym}(\mathcal{A})$ is Hermitian and satisfies
\[M_{sym} = U^{\dagger}M_{cb}(\mathcal{A})U,\]
where $U$ is the $n^k\times \binom{n+k-1}{k}$ matrix representing the isometric embedding $\mathrm{Sym}^k(\mathbb{C}^n)\hookrightarrow \mathbb{C}^{n^k}$ given by 
\[e_{\varphi(\mathbf{i})}\mapsto \frac{\sqrt{\eta(\mathbf{i})}}{k!}\sum\limits_{\sigma\in S_k}\underbrace{e_{i_{\sigma(1)}}\otimes\dots\otimes e_{i_{\sigma(k)}}}_{e_{\psi(\sigma^{-1}\cdot \mathbf{i})}}.\]
\begin{proof}
    First note that $U$ indeed represents an isometry; indeed, observe
    \begin{align*}
        \langle Ue_{\varphi(\mathbf{i})},Ue_{\varphi(\mathbf{j})}\rangle &= \frac{\sqrt{\eta(\mathbf{i})\eta(\mathbf{j})}}{(k!)^2}\sum\limits_{\sigma,\tau\in S_k}\langle e_{i_{\sigma(1)}}\otimes\dots\otimes e_{i_{\sigma(k)}},e_{j_{\tau(1)}}\otimes\dots\otimes e_{j_{\tau(k)}}\rangle  \\
        &= \frac{\sqrt{\eta(\mathbf{i})\eta(\mathbf{j})}}{(k!)^2}\sum\limits_{\sigma,\tau\in S_k}\prod\limits_{l=1}^k\langle e_{i_{\sigma(l)}},e_{j_{\tau(l)}}\rangle.
    \end{align*}
    $\prod\limits_{l=1}^k\langle e_{i_{\sigma(l)}},e_{j_{\tau(l)}}\rangle = 0$ unless $i_{\sigma(l)} = j_{\tau(l)}$ for each $l\in [k]$, hence the sum above is trivial unless $\sigma^{-1}\cdot \mathbf{i} = \tau^{-1}\cdot \mathbf{j} \iff \mathbf{i} = (\sigma\tau^{-1})\cdot \mathbf{j} \iff \mathbf{i} = \mathbf{j}$, since $\mathbf{i},\mathbf{j}\in \mathcal{I}_k$. Furthermore, when $\mathbf{i}=\mathbf{j}$, $\langle i_{\sigma(l)},i_{\tau(l)}\rangle=1$ for every $l\in [k]$ if and only if $\sigma^{-1}\cdot \sigma = \tau^{-1}\cdot \tau \iff \mathbf{i} = (\sigma\tau^{-1})\cdot \mathbf{i}$, implying that $\sigma\tau^{-1}\in \mathrm{Stab}_{S_k}(\mathbf{i})$. Therefore, 
    \begin{align*}
        \langle Ue_{\varphi(\mathbf{i})},Ue_{\varphi(\mathbf{i})}\rangle &= \frac{\eta(\mathbf{i})}{(k!)^2}\sum\limits_{\sigma,\tau\in S_k}\prod\limits_{l=1}^k\langle e_{i_{\sigma(l)}},e_{i_{\tau(l)}}\rangle \\
        &= \frac{1}{k!\mathrm{Stab}_{S_k}(\mathbf{i})}\sum\limits_{\sigma\in S_k}\left(\sum\limits_{\substack{\tau\in S_k\text{ and } \\ \sigma\tau^{-1}\in \mathrm{Stab}_{S_k}(\mathbf{i})}}1\right),\quad \text{by orbit-stabilizer theorem} \\
        &= \frac{1}{k!\mathrm{Stab}_{S_k}(\mathbf{i})}\sum\limits_{\sigma\in S_k}|\mathrm{Stab}_{S_k}(\mathbf{i})| \\
        &= 1.
    \end{align*}
    Thus, $U$ is in fact an isometry. 

    Next, observe that 
    \begin{align*}
        U\big(M_{sym}(\mathcal{A})\big)U^T &= U\left(\sum\limits_{\mathbf{i},\mathbf{j}\in \mathcal{I}_k}\sqrt{\eta(\mathbf{i})\eta(\mathbf{j})}a_{\mathbf{i}\mathbf{j}}e_{\varphi(\mathbf{i})}\circ e_{\varphi(\mathbf{j})}\right)U^T \\
        &= \sum\limits_{\mathbf{i},\mathbf{j}\in \mathcal{I}_k}\sqrt{\eta(\mathbf{i})\eta(\mathbf{j})}a_{\mathbf{i}\mathbf{j}}(Ue_{\varphi(\mathbf{i})})\circ (Ue_{\varphi(\mathbf{j})}) \\
        &= \sum\limits_{\mathbf{i},\mathbf{j}\in \mathcal{I}_k}\sqrt{\eta(\mathbf{i})\eta(\mathbf{j})}a_{\mathbf{i}\mathbf{j}}\left(\frac{\sqrt{\eta(\mathbf{i})}}{k!}\sum\limits_{\sigma\in S_k}e_{\psi(\sigma^{-1}\cdot\mathbf{i})}\right)\circ \left(\frac{\sqrt{\eta(\mathbf{j})}}{k!}\sum\limits_{\tau\in S_k}e_{\psi(\tau^{-1}\cdot\mathbf{j})}\right) \\
        &= \sum\limits_{\mathbf{i},\mathbf{j}\in \mathcal{I}_k}\frac{\eta(\mathbf{i})\eta(\mathbf{j})}{(k!)^2}a_{\mathbf{i}\mathbf{j}}\left(\mathrm{Stab}_{S_k}(\mathbf{i})\sum\limits_{\boldsymbol{\ell}\in \mathrm{Orb}_{S_k}(\mathbf{i})}e_{\psi(\boldsymbol{\ell})}\right)\circ \left(\mathrm{Stab}_{S_k}(\mathbf{j})\sum\limits_{\mathbf{m}\in \mathrm{Orb}_{S_k}(\mathbf{i})}e_{\psi(\mathbf{m})}\right) \\
        &= \sum\limits_{\mathbf{i},\mathbf{j}\in\mathcal{I}_k}a_{\mathbf{i}\mathbf{j}}\left(\sum\limits_{\substack{\boldsymbol{\ell}\in \mathrm{Orb}_{S_k}(\mathbf{i}) \\ \mathbf{m}\in \mathrm{Orb}_{S_k}(\mathbf{j})}}e_{\psi(\boldsymbol{\ell})}\circ e_{\psi(\mathbf{m})}\right),\quad \text{by orbit-stabilizer theorem} \\
        &= \sum\limits_{\mathbf{i},\mathbf{j}\in [n]^k}a_{\mathbf{i}\mathbf{j}}e_{\psi(\mathbf{i})}\circ e_{\psi(\mathbf{j})} \\
        &= M_{cb}(\mathcal{A}).
    \end{align*}
    Moreover, since $U$ is an isometry, we have that $U^TU = I_{\binom{n+k-1}{k}}$, hence it immediately follows that $M_{sym}(\mathcal{A}) = U^TM_{cb}(\mathcal{A})U$. Additionally, since $U$ has only real entries and $M_{cb}$ is Hermitian, this implies that $\big(M_{sym}(\mathcal{A})\big)^{\dagger} = U^TM_{cb}(\mathcal{A})U = M_{sym}(\mathcal{A})$, proving that $M_{sym}(\mathcal{A})$ is Hermitian.  
\end{proof}

\section{Proofs for Facts Regarding the $t$-Product Algebra}
First we prove Fact \ref{t-conjugate transpose in frequency domain}, which recall states that for any hypermatrix $\mathcal{A}\in \mathbb{C}^{m\times n\times p}$, 
\[\mathcal{A}^H = \mathrm{ifft}_3\left(\widehat{\mathcal{A}}^{\widetilde{H}}\right).\]
\begin{proof}
    Let $\mathcal{A}\in \mathbb{C}^{m\times n\times p}$ and define 
    \[\mathcal{A}^{H_F} := \mathrm{ifft}_3\Big(\widehat{\mathcal{A}}^{\widetilde{H}}\Big).\]
    We will prove that 
    \[\mathcal{A}^{H_F} = \mathcal{A}^H.\]
    By our definition for $H_F$ and Fact \ref{bcirc property} we have that 
    \begin{align*}
        \mathrm{bcirc}(\mathcal{A}^{H_F}) &= (DFT_p^H\otimes I_n)\mathrm{diag}\left(\left(\widehat{\mathcal{A}}_{(1)}\right)^H,\left(\widehat{\mathcal{A}}_{(2)}\right)^H....,\left(\widehat{\mathcal{A}}_{(p)}\right)^H\right)(DFT_p\otimes I_,) \\
        &= (DFT_p^H\otimes I_n)\mathrm{diag}\big(\widehat{\mathcal{A}}_{(1)},\widehat{\mathcal{A}}_{(2)}...,\widehat{\mathcal{A}}_{(p)}\big)^H(DFT_p\otimes I_m) \\
        &= \left((DFT_p^H\otimes I_m)\mathrm{diag}\big(\widehat{\mathcal{A}}_{(1)},\widehat{\mathcal{A}}_{(2)},...,\widehat{\mathcal{A}}_{(p)}\big)(DFT_p\otimes I_n)\right)^H \\
        &= \mathrm{bcirc}(\mathcal{A})^H.
    \end{align*}
    On the other hand, by definition of $t$-conjugate transpose, 
    \[\mathrm{bcirc}(\mathcal{A}^H) = \left[\begin{array}{cccc}
        \left(\mathcal{A}_{(1)}\right)^H & \left(\mathcal{A}_{(2)}\right)^H & \dots & \left(\mathcal{A})_{(p)}\right)^H \\
        \left(\mathcal{A}_{(p)}\right)^H & \left(\mathcal{A}_{(1)}\right)^H & \dots & \left(\mathcal{A}_{(p-1)}\right)^H \\
        \vdots & \vdots & \ddots & \vdots \\
        \left(\mathcal{A}_{(2)}\right)^H & \left(\mathcal{A}_{(3)}\right)^H & \dots & \left(\mathcal{A}_{(1)}\right)^H
    \end{array}\right],\]
    hence 
    \[\mathrm{bcirc}(\mathcal{A}^H)^H = \left[\begin{array}{cccc}
        \mathcal{A}_{(1)} & \mathcal{A}_{(p)} & \dots & \mathcal{A}_{(2)} \\
        \mathcal{A}_{(2)} & \mathcal{A}_{(1)} & \dots & \mathcal{A}_{(3)} \\
        \vdots & \vdots & \ddots & \vdots \\
        \mathcal{A}_{(p)} & \mathcal{A}_{(p-1)} & \dots & \mathcal{A}_{(1)}
    \end{array}\right] = \mathrm{bcirc}(\mathcal{A});\]
    that is, $\mathrm{bcirc}(\mathcal{A}^H) = \mathrm{bcirc}(\mathcal{A})^H$. Thus, 
    \[\mathrm{bcirc}(\mathcal{A}^{H_F}) = \mathrm{bcirc}(\mathcal{A}^{H}),\]
    which implies that $\mathcal{A}^{H_F} = \mathcal{A}^H$ since in general $\mathrm{bcirc}:\mathbb{C}^{n\times m\times p}\rightarrow \mathbb{C}^{np\times mp}$ is injective. 
\end{proof}
Next, we prove Fact \ref{Fourier conjugation}, which recall states that for any $\mathcal{A}\in \mathbb{C}^{m\times n\times p}$, 
\[\mathrm{fft}_3\big(J(\mathcal{A})\big)_{(\ell)} = \overline{\mathrm{fft}_3(\mathcal{A})_{(\ell)}}\]
for each $\ell\in [p]$, where recall 
\[J(\mathcal{A})_{(\ell+1)} := \overline{\mathcal{A}_{\big((-\ell\bmod p)+1\big)}}\]
for $\ell=0,1,\dots,p-1$. 
\begin{proof}
    For a fixed choice of $(i,j)$, let $a = (a_0,a_1,...,a_{p-1})$ with $a_{\ell} := \mathcal{A}[i,j,\ell+1]$ for $\ell=0,1,\dots,p-1$, and then define $b = (b_0,b_1,...,b_{p-1})$ with $b[\ell] := \overline{a[-\ell\bmod p]}$ for each $\ell=0,1,\dots,p-1$. Then applying the FFT to $b$, for each $\ell=0,1,\dots,p-1$ we have that 
\begin{align*}
    \widehat{b}[\ell] &= \sum\limits_{m=0}^{p-1}\overline{a[-m\bmod p]}e^{-\frac{2\pi i \ell m }{p}} \\
    &= \overline{a_0}+\sum\limits_{m'=-(p-1)}^{-1}\overline{a[m'\bmod p]}e^{\frac{2\pi i \ell m'}{p}},\quad \text{setting }m':=-m \\
    &= \overline{a_0}+\sum\limits_{m'=1}^{p-1}\overline{a[m']}e^{\frac{2\pi i \ell m'}{p}},\quad \text{by cyclicity} \\
    &= \overline{\sum\limits_{m'=0}^{p-1}a[m']e^{-\frac{2\pi i \ell m'}{p}}} \\
    &= \overline{\widehat{a}[\ell]}. 
\end{align*}
Consequently, 
\[\widehat{J(\mathcal{A})}_{(\ell)} = \overline{\widehat{\mathcal{A}}_{(\ell)}}\]
for each $\ell\in [p]$.
\end{proof}

\bibliographystyle{ieeetr}
\bibliography{references}

\end{document}